\documentclass{amsart}

\usepackage{amsmath, amsthm, amssymb, latexsym,url}

\usepackage{hyperref}

\usepackage[final]{showlabels}

\newtheorem{thm}{Theorem}[section]
\newtheorem{cor}[thm]{Corollary}
\newtheorem{rem}[thm]{Remark}
\newtheorem{lem}[thm]{Lemma}

\newtheorem{prop}[thm]{Proposition}

\newcommand{\NCF}{\mathcal{N}_{\text{CF}}}

\title[CF-normal, absolutely abnormal numbers]{On the Borel complexity of continued fraction normal, absolutely abnormal numbers}
\author{S. Jackson, B. Mance, and J. Vandehey }
\date{\today}

\begin{document}

\maketitle

\begin{abstract}
    We show that normality for continued fractions expansions and normality for base-$b$ expansions are maximally logically separate. In particular, the set of numbers that are normal with respect to the continued fraction expansion but not base-$b$ normal for a fixed $b\ge 2$ is $D_2(\boldsymbol{\Pi}_3^0)$-complete. Moreover, the set of numbers that are normal with respect to the continued fraction expansion but not normal to \emph{any} base-$b$ expansion is $D_2(\boldsymbol{\Pi}_3^0)$-hard, confirming the existence of uncountably many such numbers, which was previously only known assuming the generalized Riemann hypothesis.
    
    By varying the method of proof we are also able to show that the set of base-$2$ normal, base-$3$ non-normal numbers is also $D_2(\boldsymbol{\Pi}_3^0)$-complete. We also prove an auxiliary  result on the normality properties of the continued fraction expansions of fractions with a fixed denominator. 
\end{abstract}

\section{Introduction}

Digital systems, such as base-$b$ expansions or continued fraction (CF) expansions, associate to a given number $x\in[0,1)$ a word $w=w(x)=a_1a_2a_3\dots$ of digits.\footnote{If we ignore rational numbers, then we can guarantee all CF expansions are infinite.} Each of these expansions is associated to a transformation, $T_b(x) = bx\pmod{1}$ for base-$b$ expansions and $T_{\text{CF}}(x) = 1/x\pmod{1}$ for CF expansions, that acts as a forward shift on the word of digits. These transformations preserve and are ergodic with respect to the Lebesgue measure $\lambda$ and the Gauss measure $\mu$, respectively, where 
\[
\mu(A) = \int_A \frac{dx}{(1+x)\log 2}.
\]
Generic points for these transformations are often referred to as \emph{normal} in their respective digital systems. Given the way these transformations behave on the word of digits, normality is most often defined as a digital property. Namely, let $N_u(x,n)$ denote the number of times the finite word $u$ appears in the first $n$ digits of $x$, that is, in $a_1a_2\dots a_n$, and for a finite word $u$, let $m(u)$ denote the measure (either Lebesgue or Gauss as appropriate) of the set of $x$ where $w(x)$ begins with $u$. Then $x$ is normal if for every finite word $u$, we have that
\[
\lim_{n\to \infty} \frac{N_u(x,n)}{n} = m(u).
\]

While several constructions of base-$b$ normal numbers \cite{Besicovitch,Champernowne,CE,DE,DKK,MTT,V2013} and CF normal numbers \cite{AKS,V16a} exist (see also \cite{BMV,MM} for examples in other systems), the general problem of determining whether a given number is normal is often intractable. Because of this, and because normality is a fundamental dynamic property, much recent attention has been given to the set of normal numbers instead of individual normal numbers. For convenience, we will denote the set of base-$b$ normal numbers by $\mathcal{N}_b$ and the set of CF normal numbers by $\NCF$. By the ergodicity of the associated shift map and the pointwise ergodic theorem, $\NCF$ and $\mathcal{N}_b$ for all $b\ge 2$ are of full Lebesgue measure. 

We consider two central questions: given two different digital systems, are the corresponding sets of normal numbers different, and when they are different, how complicated are the sets of normal numbers relative to one another? In the same way that we expect a randomly chosen number to be normal, unless it has some underlying structure connected to the digital system, we also expect two different digital systems to have different sets of normal numbers (that are in some sense maximally distinct), unless the two systems have some shared underlying structure.

In practice, it is far easier to identify the shared underlying structure than it is to show that no shared structure exists. So, for instance, it is known that $\mathcal{N}_b=\mathcal{N}_{b^k}$ for all $b\ge 2$, $k\in \mathbb{N}$ \cite{Maxfield}, and it is known that $\NCF$ is the same as the set of normal numbers for certain continued fraction variants \cite{KNnormal,V2014}. Some other examples of different digital systems with the same set of normal numbers can be found in  \cite{KNSchweigerProblem,Schweiger}. On the other hand, few examples of systems with different sets of normal numbers are known. Using an intricate product measure argument, it was shown that if $b,c\ge 2$ \emph{cannot} be written as $b^r=c^s$ for integers $r,s$, then $\mathcal{N}_b\setminus \mathcal{N}_c$ is uncountable \cite{Cassels,Schmidt} (see also \cite{Schmidt62}). Even less is known about how base-$b$ normality relates to other forms of normality. A result of the third author \cite{V16} states that the set $\NCF\setminus \bigcup_{b\ge 2} \mathcal{N}_b$ of CF-normal and absolutely abnormal---i.e., not normal to any base---numbers is uncountable, but the proof is conditional on the generalized Riemann hypothesis.

One consequence of the results of this paper will provide the following unconditional result:
\begin{cor}\label{cor:main}
The set $\NCF\setminus \bigcup_{b\ge 2} \mathcal{N}_b$ of CF-normal and absolutely abnormal numbers is uncountable. Moreover, for any prime $p$, 
\[
\left( \NCF \cap \bigcup_{b=p^k} \mathcal{N}_b \right) \setminus  \bigcup_{\substack{b\ge 2\\ b\neq p^k}} \mathcal{N}_b \text{ is uncountable,}
\] so there exist uncountably many numbers which are CF-normal and normal to every base which is a power of $p$, but not normal for any other base.
\end{cor}

Knowing that the set $\NCF\setminus \bigcup_{b\ge 2} \mathcal{N}_b$ is non-empty, we can consider various ways to measure the size and complexity of this set. In this paper, we opt to consider complexity from the perspective of descriptive set theory and the Borel hierarchy.

We will briefly recall the definitions of the (boldface) Borel hierarchy. We let $\boldsymbol{\Pi}_1^0$ (and $\boldsymbol{\Sigma}_1^0$) denote the class of all closed (and open, respectively) subsets of $\mathbb{R}$. We then iteratively define $\boldsymbol{\Pi}_\alpha^0$ (and $\boldsymbol{\Sigma}_{\alpha}^0$) to be the countable intersection (union, resp.) of sets belonging to $\boldsymbol{\Sigma}_{\beta}^0$ ($\boldsymbol{\Pi}_{\beta}^0$, resp.) with $\beta<\alpha$. Since $\mathbb{R}$ is a metric space, we have for all $\alpha$ that $\boldsymbol{\Pi}_\alpha^0 \subset \boldsymbol{\Pi}_{\alpha+1}^0$ and $\boldsymbol{\Sigma}_\alpha^0 \subset \boldsymbol{\Sigma}_{\alpha+1}^0$; moreover, since $\mathbb{R}$ is an uncountable Polish space, these subset inclusions are strict. We will also consider the difference hierarchy and define the class 
\[
D_2(\boldsymbol{\Pi}_\alpha^0):= \left\{ A\setminus B : A,B\in \boldsymbol{\Pi}_{\alpha}^0 \right\}
\]
and define $D_2(\boldsymbol{\Sigma}_{\alpha}^0)$ similarly.\footnote{One can also define $D_\beta(\boldsymbol{\Pi}_{\alpha}^0)$ more generally, but this will not be relevant to this paper.} The sets $D_2(\boldsymbol{\Pi}_\alpha^0)$ and $D_2(\boldsymbol{\Sigma}_\alpha^0)$ contain the sets $\boldsymbol{\Pi}_{\alpha}^0$ and   $\boldsymbol{\Sigma}_{\alpha}^0$ and are contained in the sets $\boldsymbol{\Pi}_{\alpha+1}^0$ and   $\boldsymbol{\Sigma}_{\alpha+1}^0$, so that the difference sets live between the levels of the Borel hierarchy. All the classes above are pointclasses, that is, they are closed under inverse images by continuous functions. For a given class $\Gamma$ in the hierarchy, we say that a subset $A\subset\mathbb{R}$ is $\Gamma$-hard if for every set $B$ in a (zero-dimensional) Polish space that belongs to class $\Gamma$, there exists a continuous function $f$ such that $B=f^{-1}(A)$. Since all the classes are pointclasses, this implies that the first appearance of $A$ in the hierarchy must be in $\Gamma$ or in a higher level. Functionally, $A$ is \emph{no simpler} than $\Gamma$. We then say that $A$ is $\Gamma$-complete if $A\in \Gamma$ and if $A$ is $\Gamma$-hard.

The notion of a set being $\Gamma$-complete can be interpreted as giving us the minimum number of integer quantifiers needed to define the set. For example, it is known that $\mathcal{N}_b$, for a fixed $b\ge 2$, is $\boldsymbol{\Pi}_3^0$-complete \cite{KL}. So consider the following definition of $\mathcal{N}_b$:
\begin{quote}
    The set $\mathcal{N}_b$ consists of all numbers $x$ such that for all finite strings $s$ of base-$b$ digits and all rational $\epsilon>0$, there exists $n_0\in \mathbb{N}$ such that for all $n\ge n_0$, we have 
    \[
    \left| \frac{N_s(x,n)}{n}- \frac{1}{b^{|s|}}\right|< \epsilon.
    \]
\end{quote}
This uses one\footnote{Technically there are two universal quantifiers, one for $s$ and one for $\epsilon$, but as these are consecutive, they are treated as a single quantifier} universal, one existential, and one universal quantifier, for a total of three quantifiers matching the $3$ in the subscript of $\boldsymbol{\Pi}_3^0$, and is thus the simplest possible logical description of this set. 

Similarly, suppose $A,B$ are both $\Gamma$-complete and $A\setminus B$ is $D_2(\Gamma)$-complete. If $A\cap B \subset C$ and $C\cap A\subset C\cap B$, then $C$ must be $\Gamma$-hard. This suggests that the sets $A$ and $B$ are as logically separate as it is possible for them to be: not only does belonging to $A$ not imply belonging to $B$, but any new condition which taken together with belonging to $A$ \emph{would} imply belonging to $B$ must be at least as logically complicated as belonging to $B$ in the first place.


Starting most notably with Ki and Linton \cite{KL}, several authors have conducted research into normal numbers from the perspective of descriptive set theory. The following facts are known:
\begin{itemize}
    \item The set $\NCF$ of CF-normal numbers is $\boldsymbol{\Pi}_3^0$-complete \cite{AireyEtAl}.
    \item The set $\bigcap_{b\ge 2} \mathcal{N}_b$ of absolutely normal numbers is $\boldsymbol{\Pi}_3^0$-complete \cite{BecherEtAl}.
    \item The set $\bigcup_{b\ge 2} \mathcal{N}_b$ of numbers which are normal to at least one base is $\boldsymbol{\Sigma}_4^0$-complete \cite{BS}. 
\end{itemize}
Further results of this type can be found in \cite{AMJ,BeBe,Beros}.

The primary goal of this paper is to prove the following result.
\begin{thm}\label{thm:main}
The set $\NCF\setminus \bigcup_{b\ge 2} \mathcal{N}_b$ of numbers that are normal with respect to the continued fraction expansion but not normal with respect to any base-$b$ expansion is  $D_2(\boldsymbol{\Pi}_3^0)$-hard. Moreover, for any fixed base $b\ge 2$, the set $\NCF\setminus \mathcal{N}_b$ of numbers that are normal with respect to the continued fraction expansion but not base-$b$ normal is $D_2(\boldsymbol{\Pi}_3^0)$-complete.
\end{thm}

Since any countable set can be written as a countable union of singleton sets and is thus in $\boldsymbol{\Sigma}_2^0$, the above theorem implies the first half of Corollary \ref{cor:main} immediately. The second half is implicit in the method of proof.

Our above remarks on $D_2(\Gamma)$-completeness together with Theorem \ref{thm:main} show that there is no simple condition which can be combined CF-normality to imply base-$b$ normality. For example, base-$b$ richness is a weaker property than base-$b$ normality: a number $x$ is base-$b$ rich if every possible finite word appears at least once in $w(x)$. The set of base-$b$ rich numbers exists in $\Pi_2^0$, so therefore CF-normality and base-$b$ richness is not enough to imply base-$b$ normality.

We note that Theorem \ref{thm:main} is one-directional. To show that $\NCF\setminus \mathcal{N}_2$ (for example) is $D_2(\boldsymbol{\Pi}_3^0)$-complete, we use the fact that the typical dyadic rational has a finite CF expansion that is close to normal, but an infinite base-$2$ expansion that is very far from normal. To show that $\mathcal{N}_2\setminus \NCF$ is also $D_2(\boldsymbol{\Pi}_3^0)$-complete, we need a similarly large class of numbers that are close to base-$2$ normal but far from CF-normal. Quadratic irrationals, such as $\sqrt{3}$, are likely candidates, as they have a periodic and thus highly non-normal CF-expansion, but whether any of these quadratic irrationals is base-$2$ normal is still an open question. There are some known examples of base-$2$ normal, CF-non-normal numbers \cite{Korobov90}, so the set $\mathcal{N}_2\setminus \NCF$ is known to be non-empty; however, these examples do not have the flexibility required.

If we examine the proof of Theorem \ref{thm:main} closely and ignore any question of CF-normality, we see that we are constructing numbers that are normal to certain bases and non-normal to other bases. While the proof of Theorem \ref{thm:main} does not immediately tell use anything about the set-theoretic complexity of $\mathcal{N}_2\setminus \mathcal{N}_3$, we can vary the method to obtain the following result.

\begin{thm}\label{thm:secondary}
If $b,c\ge 2$ are relatively prime integers, then the set $\mathcal{N}_b\setminus \mathcal{N}_c$ of numbers that are base-$b$ normal but base-$c$ non-normal is $D_2(\boldsymbol{\Pi}_3^0)$-complete.
\end{thm}

We note again that $\mathcal{N}_b\neq \mathcal{N}_{c}$ if $b^r \neq c^s$ for any positive integers $r,s$. This is a much weaker condition than requiring $b,c$ to be relatively prime. The reason for our different assumption is due to a technicality of our proof. The way we prove base-$b$ non-normality is to have extremely long strings of consecutive $0$'s or $(b-1)$'s inside the base-$b$ expansion---in other words, we want to construct numbers which are extremely well approximated by $b$-adic rationals. While $b$-adic rationals tend to be close to $c$-normal expansions when $b$ and $c$ are relatively prime, this is not the case when $c$ shares a factor in common with $b$. 

A similar phenomenon has been seen before. Originally, Korobov-Stoneham numbers like
\[
\sum_{k=1}^\infty \frac{1}{3^k 2^{3^k}}
\]
were constructed to be base-$2$ normal, but it was later discovered that these numbers are easily seen to not be base-6 normal \cite{BB} (see also the work of Wagner \cite{Wagner}).  Essentially each term $\frac{1}{3^k2^{3^k}}$ has a finite base-6 expansion, but an infinite (and reasonably normal) base-$2$ expansion.

In proving the above results, we also needed to prove a result on continued fraction expansions that may be of independent interest. While normality is a property applied to infinite words, many constructions of normal numbers (such as those here) rely on concatenating finite words and so we want some way to measure how ``normal-like" a finite word is. The typical tool for this is some variant of $(\epsilon,u)$-normality, where a finite word is said to be $(\epsilon,u)$-normal if the word $u$ appears in the string to within $\epsilon$ of the desired frequency $m(u)$. (Fuller definitions will be provided later.) While this concept for base-$b$ expansions was first introduced by Besicovitch \cite{Besicovitch}, it was refined and popularized by Copeland and Erd\H{o}s \cite{CE}. In particular, Copeland and Erd\H{o}s showed that the number of length $k$ strings in base-$b$ that are \emph{not} $(\epsilon,u)$-normal is at most $b^{k(1-\delta)}$, where $\delta$ only depends on $\epsilon$ and the length of $u$. 

To prove the above results we needed a Copeland-Erd\H{o}s-type result for continued fractions. While certain results of this type are known (indeed, this question is closely related to the study of Gauss-Kuzmin statistics), none of the exact type we needed were. Since finite blocks of continued fraction digits are the entire CF expansion of a rational number, one often looks at a set of interesting rationals and asks how many of these fail to be $(\epsilon,u)$-normal for some word of digits $u$. What we needed for the above results was bound on how many rationals in $[0,1)$ \emph{with a fixed denominator} are not $(\epsilon,u)$-normal.

\begin{thm}\label{thm:extra}
Let $u$ be a fixed block of continued fraction digits, let $d$ be a sufficiently large positive integer, and let $\epsilon>0$ be a small positive number. Then the number of rational numbers $a/d$ in $(0,1)$ with $a$ relatively prime to $d$ with a continued fraction expansion that is not $(\epsilon,u)$-normal is at most 
\[
O\left(d^{1-\eta/\log\log d }\right),
\]
for some constant $\eta$ depending only on $\epsilon$ and $u$.
\end{thm}

In addition to being compared to Copeland and Erd\H{o}s, this theorem may be considered a complimentary result to the work of Bykovskii-Frolenkov \cite{BF} and Ustinov \cite{Ustinov}. Those results could be used to give the relative frequency of the word $u$ averaged over all rational numbers $a/d$ to a high degree of accuracy; however, they leave open the possibility that most fractions $a/d$ see the word $u$ with a frequency much higher or lower than the expected frequency. In one sense, then, Theorem \ref{thm:extra} may be seen as a loose bound on the variance.

\subsection{Strings, blocks, and words}

For any set $\mathcal{D}$, we let $\mathcal{D}^*=\bigcup_{k=1}^\infty \mathcal{D}^k$ denote the set of all finite sequences on $\mathcal{D}$. These sequences are generally known as strings, words, or blocks. Because we are working in several different systems simultaneously, from this point forward, we will use the varying terminologies to help us understand which setting we are in. We will use string (with variables like $s,t$) to denote finite sequences of base-$b$ digits. We will use block (with variables like $A,B$) to denote finite sequences of CF digits. And we will use word (with variables like $u,v,w$) if what we are discussing could consist of either base-$b$ digits or CF digits. We will denote collections of words, blocks, or strings by using calligraphic letters.

If $w\in \mathcal{D}^k$, then we say that the length of $w$ is $k$, and this is denoted by $|w|=k$. The empty word (which has length $0$) is denoted $\wedge$.

If $u=a_1a_2\dots a_n$ and $v=c_1c_2\dots c_m$, then we denote concatenation by $uv= a_1\dots a_n c_1\dots c_m$. We will say that $u$ is a prefix of $v$ if there exists another (possibly empty) word $w$ such that $uw=v$. Likewise, $u$ is a suffix of $v$ if there exists a (possibly empty) $w$ such that $wu=v$.

When we wish to concatenate a word with itself we will use exponents. So $w^k$ means the concatenation of $k$ copies of $w$. If we are repeatedly concatenating the same digit, we will use parentheses to avoid ambiguity: $(2)^k$ means a word consisting of $k$ repetitions of the digit $2$, whereas $2^k$ indicates the number $2$ raised to the power $k$.

Note that for any base $b\ge 2$, we can write a positive integer in $[b^k,b^{k+1})$ as a $k$-length string of base-$b$ digits in the usual way. In fact, any integer in $[0,b^{k+1})$ can be written as a string of base-$b$ digits of any length that is at least $k$ by prepending enough $0$'s, and we will make sure the intended length is made clear by context. As such we will often go back and forth between treating positive integers as strings or as integers as the situation requires.

\section{An outline of the proof of Theorem \ref{thm:main}}\label{sec:outline}

As the proof of the main theorem is quite intricate, we provide an outline of it here.

To begin with, we will describe in general how one may construct an example of a CF-normal but absolutely abnormal number. Suppose that $a_1a_2a_3\dots$ is the CF expansion for some real number $x\in [0,1)$. The rational numbers whose finite CF expansion is a truncation of the expansion for $x$ are known as the convergents of $x$. Suppose that a convergent of $x$ is a $b$-adic rational  $\frac{c}{b^k}$ for some integers $c$ and $k$. By the well-known relation of a number to its CF convergents, we have that
\[
\left| x- \frac{c}{b^k} \right| \le \frac{1}{b^{2k}}.
\]
Thus, since $\frac{c}{b^k}$ has only $k$ base-$b$ digits (when written with terminating $0$'s) we have that from roughly the $(k+1)$st base-$b$ digit to the $(2k)$th base-$b$ digit of $x$, we should see either a long string consisting of just the digit $0$ or just the digit $(b-1)$. In particular, the first $2k$ base-$b$ digits of $x$ look very far from normal. If infinitely many convergents of $x$ are $b$-adic rationals, then $x$ cannot be base-$b$ normal. A similar idea was used by Martin \cite{Martin} to construct an explicit example of an irrational, absolutely abnormal number.

In \cite{V16}, the third author constructed CF-normal, absolutely abnormal numbers by starting with a known CF-normal number $x$ and varying its digits slightly to cause it to have $b$-adic rational convergents infinitely often. The construction was conditional because in order to show that one could reach a $b$-adic rational convergent with a small number of changes, one had to know that there were primes in an arithmetic progression with a prescribed primitive root, which necessitated assuming GRH. 

However, we need a different method of construction to show that $\NCF\setminus \bigcup_{b\ge 2} \mathcal{N}_2$ is  $D_2(\boldsymbol{\Pi}_3^0)$-hard. We will use the technique of Wadge reduction. Namely, let $C= \{z\in \omega^\omega: z(2n+1)\to \infty\}$ and $D=\{z\in \omega^\omega: z(2n)\to \infty\}$, then $C\setminus D$ is known to be $D_2(\boldsymbol{\Pi}_3^0)$-complete. To prove that $\NCF\setminus \bigcup_{b\ge 2} \mathcal{N}_2$ is  $D_2(\boldsymbol{\Pi}_3^0)$-hard it suffices to construct a continuous map $\phi:\omega^\omega\to [0,1)$ such that $\phi^{-1}(\NCF\setminus \bigcup_{b\ge 2} \mathcal{N}_2)=C\setminus D$.

We will construct the function $\phi$ by considering an arbitrary element $z\in \omega^\omega$ and building the CF expansion of $\phi(z)$ iteratively. We consider $\Delta= \{(i,j):i\in\mathbb{N}, 1\le j \le i\}$ to be a sequence ordered lexicographically. For each $(i,j)\in \Delta$ we will construct a block $\overline{B}_{i,j}$, which will be concatenated in order to produce the CF expansion of $\phi(z)$:
\[
\overline{B}_{1,1}\overline{B}_{2,1}\overline{B}_{2,2}\overline{B}_{3,1}\overline{B}_{3,2}\overline{B}_{3,3}\dots.
\]
A given block $\overline{B}_{i,j}$ will depend only on the prior blocks, as well as $z(2i)$ and $z(2i-1)$. In this way, if $z$ and $z'$ have the same prefix, we see that the corresponding CF expansions of $\phi(z)$ and $\phi(z')$ have the same prefix as well, and thus $\phi$ is continuous. 

When $j\ge 2$, the block $\overline{B}_{i,j}$ will be constructed so that $\overline{B}_{1,1}\overline{B}_{2,1}\dots \overline{B}_{i,j}$ contains the CF expansion of a $p_j$-adic rational, where $p_j$ is the $j$th prime number, and the next CF digit is extremely large. In a refinement of the idea from \cite{V16}, we show that this is sufficient to not only obtain a large string of $0$'s or $(p_j-1)$'s in the base-$p_j$ expansion, but also a large string of $0$'s or $(bp_j-1)$'s in the base-$bp_j$ expansion for $b\ge 2$. Since any $j$ repeats infinitely often in $\Delta$, this would immediately tell us that $\phi(z)$ is not normal to any base with an odd prime factor. Since base-$2$ normality is equivalent to base-$2^k$ normality for any $k\ge 1$, we thus need only focus on CF normality and base-$2$ normality.

To control CF normality, we will consider a sequence of positive reals $(\epsilon_{i,j})_{(i,j)\in\Delta}$ that descends to zero, a sequence of collections of finite blocks $(\mathcal{A}_{i,j})_{(i,j)\in \Delta}$ such that each collection is a subset of the next collection and eventually any finite block can be found in one of the collections, and a sequence of positive integers $(n_{i,j})_{(i,j)\in\Delta}$. We would like to have our concatenated block $\overline{B}_{i,j}$ be close to normal in the $(\epsilon_{i,j},\mathcal{A}_{i,j},n_{i,j})$-normal sense---that is, any block $A\in\mathcal{A}_{i,j}$ appears in any prefix of $\overline{B}_{i,j}$ that is a multiple of $n_{i,j}$ to within $\epsilon_{i,j}$ of $\mu(C_A)$, where $C_A$ is the set of $x$ whose CF expansion start with $A$. So each new block in the concatenation should look more and more normal for more and more strings. While concepts similar to $(\epsilon_{i,j},\mathcal{A}_{i,j})$-normality are common, we make use of $(\epsilon_{i,j},\mathcal{A}_{i,j},n_{i,j})$-normality because in our construction, the length of $\overline{B}_{i,j}$ will often vastly exceed that of prior blocks in the concatenation. If we only asked for $(\epsilon_{i,j},\mathcal{A}_{i,j})$-normality, there could be some prefix of $\overline{B}_{i,j}$ that is very far form $(\epsilon_{i,j},\mathcal{A}_{i,j})$-normal and this behavior could dominate the behavior of the preceding blocks in the concatenation. In practice, we will construct our blocks $\overline{B}_{i,j}$ to consist of several $(\epsilon_{i,j},\mathcal{A}_{i,j},n_{i,j})$-normal blocks concatenated with a long string of $1$'s, where the length of this string of $1$'s depends inversely on the size of $z(2i-1)$. As a sufficiently long string of $1$'s in the CF expansion would disrupt CF normality, we guarantee that $\phi(z)\in\NCF$ if and only if $z\in C$.

To balance the needs of the previous two paragraphs, we must be able to find a block which is the CF expansion of a $p$-adic rational while maintaining good CF normality properties. In constructing $\overline{B}_{i,j}$ we select a $d$ that is a very large power of $p_j$ and look at all fractions $a/d$ whose CF expansion is prefixed by $\tilde{B}_{i,j}:=\overline{B}_{1,1}\overline{B}_{2,1}\dots, \overline{B}_{i,j-1}$. Using an idea of Avdeeva and Bykovskii \cite{AB}, we show that if $\tilde{B}_{i,j}B$ is the CF expansion of $a/d$, then $\tilde{B}_{i,j}B$ can \emph{almost} be written as the concatenation of two blocks $B'$ and $(B'')^*$ where $B'$ and the reversal of $(B'')^*$ are the CF expansions of fractions with denominator \emph{at most} $\sqrt{d}$. While there has not been much study of CF expansions of rationals with a fixed denominator, there are many results about the CF expansions of rationals with a bounded denominator; and we, in particular, use a result of Scheerer \cite{Scheerer}, to show that for most choices of $a/d$, both $B'$ and $(B'')^*$ have very good CF normality properties in the sense of being $(\epsilon_{i,j},\mathcal{A}_{i,j},n_{i,j})$-normal. On top of these conditions, we also don't want these blocks, where we are breaking base-$p_j$ normality, to force us to also break base-$2$ normality. So we also choose our $B$ so that the binary expansion of $a/d$ is fairly close to normal as well. It is for this reason that we have been looking at prime bases rather than arbitrary bases: a $3$-adic rational is likely to have a good binary expansion, but a $6$-adic rational may not. Once we have a good choice of $B$, we can get $\overline{B}_{i,j}$ from it by appending a large digit (to guarantee that we break not only base $p_j$ normality, but any multiple of $p_j$-normality) and then the string of $1$'s as mentioned above.

The blocks $\overline{B}_{i,j}$ when $j=1$ correspond to the times when we exert control over base-$2$ normality. Here, we need to use a modification of the above scheme. If we chose $B$ so that $\tilde{B}_{i,j}B$ is the CF expansion of some $2$-adic rational $\frac{c}{2^k}$, then any CF expansion with this prefix would have a binary expansion that had only $0$'s or only $1$'s from roughly its $k$th to $2k$th digits, and this would force $\phi(z)$ to be base-$2$ non-normal always. So instead of constructing $\overline{B}_{i,j}$ from the entirety of such a $B$, we instead construct it from some prefix of $B$, chosen so that it only has a ``bad" binary expansion from the $(k+1)$th to $(k+k')$th digits, where $k'$ is chosen based on $z(2i)$ and in this way allow us to base whether $\phi(z)$ is base-$2$ normal on whether $z\in D$ or not. However, in order to know that we can select a prefix of $B$ which affects a very precise amount of the binary expansion of $\phi(z)$, we must impose a further condition on $B$ that functionally says that it does not have too many abnormally large digits.

One further point to make about the above outline: in general, much of the construction is about preserving CF-normality or base-$2$ normality at each stage of the construction except at points where we very deliberately break normality in a big way. The exception to this is the string of $1$'s appended at the end of the $\overline{B}_{i,j}$'s to control CF normality. If this string is short, then it has a negligible impact on both CF-normality and base-$2$ normality. However, if this string is quite long, then it breaks CF normality while having an unknown effect on base-$2$ normality. This unknown effect is fine for the purposes of this proof, but is part of the reason why we cannot easily prove facts about $\mathcal{N}_2\setminus \NCF$.

\section{Preliminaries}

\subsection{Note on asymptotic notations}

We will make use of standard asymptotic notations. By $f(x)=O(g(x))$ we mean that there exists some constant $C$ such that $|f(x)|\le C|g(x)|$. By $f(x) \asymp g(x)$ we mean that there exists some constants $c,C$ such that $c|g(x)|\le |f(x)|\le C|g(x)|$. The constants $c,C$ are said to be the implicit constants.

\subsection{Facts about base-$b$ expansions}

The following lemma will allow us to describe the extent to which  an interval, usually defined by a finite CF expansion, also defines a corresponding base-$b$ expansion.

\begin{lem}\label{lem:base-b expansion of I}
Let $b\ge 2$ be an integer base and let $I\subset[0,1)$ be an interval. Letting $k=-\lceil \log_b \lambda(I)\rceil$, there exists $c\in \mathbb{N}_{\ge 0}$ such that 
\[
 I \subset \left[ \frac{c}{b^k}, \frac{c+2}{b^k}\right).
\]
\end{lem}

\begin{proof}
Note that the definition of $k$ implies that $\lambda(I) \le b^{-k}$, so either $I$ is contained in a single interval of the form $[\frac{a}{b^k}, \frac{a+1}{b^k})$, or it is contained in the union of two adjacent such intervals. This proves the existence of $c$.
\end{proof}

Since the above result will be used frequently in many contexts in this paper, we will give some more notation. First, we let $L_b(I)=-\lceil \log_b \lambda(I)\rceil$. This tells us roughly how many base-$b$ digits are predetermined if we know a point $x$ belongs to $I$. Secondly, we will refer to $c$ by $S_b(I)$. This tells us roughly what base-$b$ digits are predetermined. We say roughly in both places because of possibilities such as $I$ being a very small interval that straddles $1/2$, so that on one side, it contains points of the form $0.011111\dots 1$ in base-$2$, and on the other side it contains points of the form $0.100\dots 0$ in base-$2$; however, in this paper, any such interval we consider will be assumed to have reasonably good base-$b$ normality properties, so while this behavior could happen at the end of $S_b(I)$, the number of digits this will alter will be negligible compared to $L_b(I)$. 

\begin{rem}\label{rem:base-b digit variants}
If $x\in I$, as in Lemma \ref{lem:base-b expansion of I}, then the first $L_b(I)$ base-$b$ digits of $x$ are either $S_b(I)$ or $S_b(I)+1$. If $x$ is actually equal to $c/b^k$ in Lemma \ref{lem:base-b expansion of I}, then there is a way to write $x$ such that the first $L_b(I)$ base-$b$ digits of $x$ are $S_b(I)-1$, but this will not be relevant in most parts of the paper.
\end{rem}

\subsection{Facts about continued fractions}

The continued fraction expansion of a number $x\in [0,1)$ is a way of writing $x$ in one of the two following ways:
\[ 
x= \cfrac{1}{a_1+\cfrac{1}{a_2+\cfrac{1}{a_3+\dots}}}, \text{ if }x\in[0,1)\setminus \mathbb{Q}
\]
or
\[
x= \cfrac{1}{a_1+\cfrac{1}{a_2+\cfrac{1}{a_3+\dots+\cfrac{1}{a_k}}}}, \text{ if } x\in [0,1)\cap \mathbb{Q}
\]
where each $a_i$ as well as $k$ is a positive integer. The block $a_1a_2a_3\dots $ (or $a_1a_2a_3\dots a_k$ in the second case) is what we will generally be referring to as the CF expansion of a number.

If $B\in\mathbb{N}^*$ is a finite-length block, we will let $r_B$ denote the rational number whose CF expansion is $B$. The map $r_\cdot: \mathbb{N}^*\to \mathbb{Q}\cap (0,1]$ is a two-to-one map everywhere with a single exception\footnote{The exception is at $1$, because there is only one finite block $B=1$ that has $r_B=1$.}. This is because if $B=a_1a_2 \dots a_n$ with $a_n>1$ and $B'=a_1a_2\dots (a_n-1)1$, then $r_B=r_{B'}$. In contrast, the map from infinite sequences to irrational numbers is a bijection.

We will sometimes use an alternate way of writing continued fractions. If $B=a_1a_2\dots a_n$. Then we will sometimes write $r_B$ as $[a_1,a_2,\dots, a_n]$.

Given a (possibly infinite) block $B\in\mathbb{N}^*\cup \mathbb{N}^\omega$, we will let $\frac{p_n}{q_n}=\frac{p_n(B)}{q_n(B)}$ denote the fraction $r_{a_1a_2\dots a_n}$ in lowest terms. (If $B$ is finite, we let $p(B)$ and $q(B)$ denote $p_{|B|}(B)$ and $q_{|B|}(B)$, respectively.)  These $p_n$'s and $q_n$'s obey the following recurrence relation:
\[
p_{n+1} = a_{n+1} p_n + p_{n-1} \qquad q_{n+1}= a_{n+1}q_n+q_{n-1}, \qquad n\ge 0,
\] where we have implicitly defined
\begin{align*}
    p_{-1} &=1 & p_0&= 0\\
    q_{-1}&= 0 & q_0 &= 1.
\end{align*}
For consistency, we will define $r_\wedge=\frac{0}{1}$.

In particular, since $a_i\ge 1$ for all $i$, we have that
\[
q_{n+2}\ge q_{n+1}+q_n \ge (q_n+q_{n-1})+q_n \ge 2q_n, \qquad n\ge 0,
\]
and, so
\begin{equation}\label{eq:qn relation}
    q_{n+2}\ge 2q_n, \qquad n\ge 0.
\end{equation} 
Moreover, this inequality is strict provided $n\ge 1$.

The above recurrence relations and initial conditions also imply the following matrix relation quite easily:
\begin{equation}
\left( \begin{array}{cc} p_{n-1} & p_n \\ q_{n-1} & q_n \end{array} \right) = \left( \begin{array}{cc} 0 & 1 \\ 1 & a_1 \end{array} \right) \left( \begin{array}{cc} 0 & 1 \\ 1 & a_2 \end{array} \right) \dots \left( \begin{array}{cc} 0 & 1 \\ 1 & a_n \end{array} \right) . \label{eq:CF matrix relation}
\end{equation}And then this matrix relation implies that 
\begin{equation}
p_{n-1}q_n-p_nq_{n-1} = (-1)^n \label{eq:CF determinant relation}
\end{equation}
by taking determinants.

Suppose $B=a_1a_2\dots a_n$. Let $B^*=a_na_{n-1}\dots a_1$ denote its reversal. Then by applying the transpose to the matrix relation in \eqref{eq:CF matrix relation}, we have that if
\[
r_B = \frac{p}{q} \text{ in lowest terms, then }r_{B^*} = \frac{p^*}{q}
\]
where, by \eqref{eq:CF determinant relation}, we see that $p^*$ is the unique integer in $[1,q]$ satisfying
\[
pp^*+(-1)^{|B|}\equiv 0 \pmod{q}.
\]

We will let $r_B(m)$ denote $r_{B'}$ where $B'$ is the longest prefix of $B$ with $q(B')\le m$. In other words, this is the last convergent whose denominator does not exceed $m$. If $a_1(B)>m$, then the longest such prefix is the empty word $\wedge$, so that $r_B(m)=\frac{0}{1}$ in this case.

The following lemma will be crucial. It will allow us to compare all CF expansions with a fixed denominator $d$ with CF expansions of numbers with denominator \emph{at most} $m\approx \sqrt{d}$.

\begin{lem}\label{lem:splitter}
Consider $B\in \mathbb{N}^*$. Suppose we have integers $b\ge 2$, $m\ge 1$ satisfying
\[
\frac{q(B)}{b} \le m^2 < q(B).
\]
Let $B_1,B_2$ be a prefix and suffix of $B$ (respectively) so that $r_{B_1}=r_B(m)$ and $r_{B_2^*}=r_{B^*}(m)$. Then \[-4\left\lceil \frac{1}{2} \log_2 b \right\rceil -5 \le 
|B_1|+|B_2|-|B|\le 1.
\]
\end{lem}

\begin{proof}
The upper bound is a slight refinement of a result of Bykovskii \cite[Lemma 2]{Bykovskii}, and we will largely follow their method.

We make use of the following well-known fact:
\[
q_{i}(B)q_{n-i}(B^*)+ q_{i-1}(B)q_{n-i-1}(B^*)=q(B), \text{ for }0\le i \le n.
\]
In particular, it is true for all $i\in [0,n]$ that we have
\[
q_{i}(B)<\sqrt{q(B)} \text{ or } q_{n-i}(B^*) < \sqrt{q(B)}.
\]
The number of $i$ for which both of these inequalities are true (rather than just one) is at most $2$. This follows because if both inequalities are satisfied then we have that
\begin{align*}
q(B)&=q_{i}(B)q_{n-i}(B^*)+ q_{i-1}(B)q_{n-i-1}(B^*)\\
&\le 2q_{i}(B)q_{n-i}(B^*)\\
&< 2q_{i}(B) \cdot \sqrt{q(B)}.
\end{align*}
Rearranging gives us the lower bound of
\begin{equation}\label{eq:qi+1s bound}
\frac{1}{2} \sqrt{q(B)} < q_{i}(B) < \sqrt{q(B)}.
\end{equation}
Since the value of $q_{i}(B)$ must at least double in size every time $i$ is increased by $2$ (see \eqref{eq:qn relation}), we have that there are at most $2$ solutions to \eqref{eq:qi+1s bound}.

By construction we have that
\[
|B_1|+|B_2|=\#\{i\in [1,n]:q_i(B)\le m\}+ \#\{i\in[1,n]: q_{i}(B^*)\le m\}.
\]
And so
\begin{align*}
&|B_1|+|B_2|+2\\
&\qquad=\#\{i\in [0,n]:q_i(B)\le m\}+ \#\{i\in[0,n]: q_{i}(B^*)\le m\}\\
&\qquad\le \#\{i\in [0,n]:q_i(B)\le \sqrt{q(B)}\}+ \#\{i\in[0,n]: q_{i}(B^*)\le \sqrt{q(B)}\}\\
&\qquad= \#\{i\in [0,n]:q_i(B)\le \sqrt{q(B)}\}+ \#\{i\in[0,n]: q_{n-i}(B^*)\le \sqrt{q(B)}\}\\
&\qquad \le |B|+1+2.
\end{align*}
The last inequality comes from the fact that every $i\in[0,n]$ belongs to either the first or second set, and at most $2$ belong to both. This gives the upper bound.

Now let $k=\lceil \frac{1}{2}\log_2 b \rceil+1$ so that
\[
\frac{\sqrt{q(B)}}{2^k} < \frac{\sqrt{q(B)}}{\sqrt{b}}. 
\] Then every $i$ such that $q_i(B)< \sqrt{q(B)}$ must also satisfy $q_i(B)\le m$ with at most $2k$ exceptions. This again follows from the equation \eqref{eq:qn relation} that shows that $q_i(B)$ must at least double every time $i$ increases by $2$. And similarly there are at most $2k$ exceptions for $q_{n-i}(B^*)<\sqrt{q(B)}$ implying $q_{n-i}(B^*)\le m$. Therefore, we have that \begin{align*}
    &|B_1|+|B_2|+2\\
&\qquad=\#\{i\in [0,n]:q_i(B)\le m\}+ \#\{i\in[0,n]: q_{i}(B^*)\le m\}\\
&\qquad \ge \#\{i\in [0,n]:q_i(B)\le \sqrt{q(B)}\}+ \#\{i\in[0,n]: q_{i}(B^*)\le \sqrt{q(B)}\}-4k\\
&\qquad = \#\{i\in [0,n]:q_i(B)\le \sqrt{q(B)}\}+ \#\{i\in[0,n]: q_{n-i}(B^*)\le \sqrt{q(B)}\}-4k\\
&\qquad \ge |B|+1-4k.
\end{align*}
Here, the last inequality comes from the fact that every $i$ must contribute to either the first or second set. This gives the lower bound.
\end{proof}

Given a block $B\in\mathbb{N}^*$ we will let $C_B\subset[0,1)$ denote the corresponding cylinder set, the set of all $x\in[0,1)$ whose continued fraction expansion is prefixed by $B$. The rank of a cylinder $C_B$ is the length of the corresponding block $B$. Suppose  $|B|>1$ and we let $B'$ denote the prefix of $B$ with $|B|-1$ elements. Then $C_B$ is a clopen interval whose endpoints are
\[
\frac{p(B)}{q(B)} \text{ and } \frac{p(B')+p(B)}{q(B')+q(B)}.
\]
Since $p(B')q(B)-p(B)q(B')=(-1)^{|B|}$, this implies that
\begin{equation}
\lambda(C_B) = \frac{1}{(q(B')+q(B))q(B)}. \label{eq:Cs lebesgue measure}
\end{equation}

This immediately implies the following important result:

\begin{prop}\label{prop:Lebesgue to denominator comparison}
For any block $B\in\mathbb{N}^*$, we have that
\[
\lambda(C_B) \asymp q(B)^{-2}.
\]
\end{prop}

\begin{proof}
This follows by simply applying the inequality
\[
q(B)\le q(B')+q(B) \le 2q(B)
\]
to \eqref{eq:Cs lebesgue measure}.
\end{proof}

The following lemma gives a lower bound on how many base-$b$ digits are determined by a given number of CF digits.

\begin{lem}\label{lem:CF to base b digits}
For any base $b\ge 2$ and any finite block $B\in \mathbb{N}^*$ of CF digits, we have that
\[
L_b(C_B) \ge 2|B| \log_b(\phi)+O(1)
\]
where $\phi=(1+\sqrt{5})/2$
\end{lem}

\begin{proof}
Let us consider all blocks $B$ that have the same length. First note that $L_b(C_B)$ is smallest when $\lambda(C_B)$ is largest, and that, by \eqref{eq:Cs lebesgue measure}, this happens when $q(B),q(B')$ are as small as possible.  The smallest $q(B),q(B')$ could be is if $B$ is composed entirely of $1$'s. In this case, $q(B)=F_{|B|+1}$ and $q(B')=F_{|B|}$, where $F_n$ is the $n$th Fibonacci number, starting with $F_0=0$, $F_1=1$. Using Binet's formula, we see that in this case, $\lambda(C_B) \asymp \phi^{-2|B|}$. Therefore, in this worst-case scenario, we have
\begin{align*}
L_b(C_B) & = -\lceil \log_b \lambda(C_B) \rceil = - \left\lceil -2|B| \log_b( \phi) +O(1) \right\rceil\\
&= 2|B|\log_b(\phi)+O(1).
\end{align*}
This gives the desired result.
\end{proof}

\subsection{Dynamics of continued fractions}

We let $T:[0,1)\to [0,1)$ denote the usual Gauss map (which we referred to as $T_{\text{CF}}$ in the introduction):
\[
Tx = \begin{cases} \dfrac{1}{x}-\left\lfloor \dfrac{1}{x} \right\rfloor, & x\neq 0,\\
0, & x=0.
\end{cases}
\]
We recall that the Gauss map has a corresponding invariant measure, the Gauss measure $\mu$, defined by
\[
\mu(A) = \frac{1}{\log 2} \int_A \frac{dx}{1+x}.
\]
The definition easily implies that the Gauss measure and Lebesgue measure are very closely related. In fact, for any measurable set $A$, we have
\[
\frac{\lambda(A)}{2\log 2} \le \mu(A) \le \frac{\lambda(A)}{\log 2},
\]
and so
\begin{equation}
\mu(A) \asymp \lambda (A). \label{eq:mu nu asymptotic}
\end{equation}

The Gauss map satisfies Reny\'{i}'s condition, which states that for any cylinder $C_B$, we have that there is a uniform upper bound $L$ on 
\[
\frac{\sup_{z\in [0,1)} |(T_B^{-1})' z|}{\inf_{z\in [0,1)} |(T_B^{-1})' z|},
\]
where $T_B$ is the map $T^k:C_B\to [0,1)$.

Reny\'{i}'s condition says that the Gauss map has bounded distortion on each cylinder. The key consequence of this that we will use in this paper is the following: for any blocks $B,B'$, we have
\begin{equation}
\frac{1}{L} \mu(C_B)\mu(C_{B'})\le \mu(C_{B B'}) \le L \cdot \mu(C_B)\mu(C_{B'}).\label{eq:Renyi-condition}
\end{equation} To see this, we use the following argument:
\begin{align*}
    \frac{\mu(C_{BB'})}{\mu(C_B)} &= \frac{\mu\left(T_B^{-1} C_{B'} \right)}{\mu(T_B^{-1} [0,1))}= \frac{\int_{T_B^{-1} C_{B'}} d\mu}{\int_{T_B^{-1} [0,1)} d\mu}= \frac{\int_{C_{B'}} |(T_B^{-1})'z| d\mu}{\int |(T_B^{-1})'z| d\mu}\\
    &\le \frac{\int_{C_{B'}} d\mu \cdot \sup_{z\in [0,1)} |(T_B^{-1})'z|}{\int d\mu\cdot \inf_{z\in [0,1)} |(T_B^{-1})'z|} \le L \frac{\mu(C_{B'})}{1} = L\cdot  \mu(C_{B'}).
\end{align*}
The proof of the lower bound follows a similar argument.

Using Renyi's condition, we may prove a variety of related results.

\begin{prop}\label{prop:product of denominators}
Let $B,B'$ be any blocks of CF digits. Then all of the following hold with uniform implicit constants.
\begin{enumerate}
    \item $\lambda(C_{BB'}) \asymp \lambda(C_{B})\lambda(C_{B'})$
    \item $q(BB') \asymp q(B)q(B')$
    \item $L_b(C_{BB'}) = L_b(C_B)+L_b(C_{B'})+O(1)$
\end{enumerate}
\end{prop}

\begin{proof}
The first result holds by combining \eqref{eq:mu nu asymptotic} with \eqref{eq:Renyi-condition}. The second follows from the first and Proposition \ref{prop:Lebesgue to denominator comparison}. The third follows from the first, the definition of $L_b(\cdot)$, and the fact that $\lceil z\rceil = z+O(1)$.
\end{proof}

The following lemma will help us determine how much appending a string of $1$'s to a CF expansion will alter the base-$b$ expansion.

\begin{lem}\label{lem:length - adding block of 1s}
For any finite block $B$ of continued fraction digits, any integer base $b\ge 2$, and any positive integer $k$, we have that
\[
L_b(C_{B(1)^k}) = L_b(C_B)+2k \log_b(\phi)+O(1),
\]
where $\phi=(1+\sqrt{5})/2$
\end{lem}

\begin{proof} This follows immediately by combining parts (1) and (3) of Proposition \ref{prop:product of denominators} together with facts discussed in the proof of Lemma \ref{lem:CF to base b digits}.
\end{proof}

In the following we  consider how adding a single digit to a CF expansion alters the base-$b$ expansion.

\begin{lem}\label{lem:length - adding a single digit}
For any finite block $B$ of continued fraction digits, any integer base $b\ge 2$, and any positive integer $d$, we have that
\[
L_b(C_{Bd}) = L_b(C_B)+2 \log_b(d)+O(1).
\]
\end{lem}

\begin{proof}
This follows by combining part (3) of Proposition \ref{prop:product of denominators}, Proposition \ref{prop:Lebesgue to denominator comparison}, and the definition of $L_b(\cdot)$.
\end{proof}

\begin{lem}\label{lem:C_B versus C_B^*}
For any block $B$, we have $\mu(C_B)=\mu(C_{B^*})$.
\end{lem}

\begin{proof}
For this we consider the natural extension $\tilde{T}:[0,1)^2\to[0,1)^2$ of $T$, which acts bijectively by
\[
\tilde{T}(x,y)= \left(Tx, \frac{1}{a_1(x)+y}\right),
\]
when $x\neq 0$ and leaves the following measure invariant:
\[
\tilde{\mu}(A) = \frac{1}{\log 2} \iint_A \frac{dx \ dy}{(1+xy)^2}.
\]
It's also easy to see that for any $E\subset [0,1)$ that $\tilde{\mu}(E\times [0,1))=\tilde{\mu}([0,1)\times E)=\mu(E)$.
We then have by the invariance of $\tilde{T}$ that
\begin{align*}
    \mu(C_B) &= \tilde{\mu}(C_B\times [0,1)) = \tilde{\mu}\left(\tilde{T}^{|B|}(C_B\times [0,1))\right)\\
    &= \tilde{\mu}([0,1)\times C_{B^*}) = \mu(C_{B^*})
\end{align*}
as desired.
\end{proof}

\subsection{Farey fractions}

We let $\mathcal{F}_m$ denote the set of all fractions in $[0,1]$ which have a denominator at most $m$ when written in lowest terms. (We will always assume fractions are written in lowest terms.) It is known that $\#\mathcal{F}_m\asymp m^2$. The set $\mathcal{F}_m$ has a natural ordering based on $<$ and we will say that $P/Q,P'/Q'\in \mathcal{F}_m$ are consecutive fractions in $\mathcal{F}_m$ if $P/Q<P'/Q'$ and if no other element of $\mathcal{F}_m$ lies between them.

If $\frac{P}{Q},\frac{P'}{Q'}\in\mathcal{F}_m$ are consecutive, then it is well-known that 
\begin{equation}
\frac{P'}{Q'}-\frac{P}{Q} = \frac{1}{QQ'} \qquad \text{or, equivalently,} \qquad QP'-PQ'=1. \label{eq:consecutive farey denominators}
\end{equation}

Note that if $\frac{P}{Q},\frac{P'}{Q'}\in\mathcal{F}_m$ are consecutive, then it must be that $\max(Q,Q')\ge m/2$. This is true because the mediant $\frac{P+P'}{Q+Q'}$ lies between $P/Q$ and $P'/Q'$ and if both $Q,Q'< m/2$, then the mediant belongs to $\mathcal{F}_m$, so the fractions cannot be consecutive.

Suppose $\frac{P}{Q}\in \mathcal{F}_m$ with $Q\ge 2$ is such that neither neighboring fraction (the one it precedes and the one that it succeeds) has a larger denominator. Moreover, suppose the continued fraction expansion of $\frac{P}{Q}$ can be written as $[a_1,a_2,\dots, a_n+1]$ so that the last digit is not $1$. (Note that since $Q\ge 2$ this is always possible.) Then the preceding and succeeding fractions are, in some order,
\begin{equation}
[a_1,a_2,\dots, a_n] \qquad \text{ and } \qquad [a_1,a_2,\dots, a_{n-1}]. \label{eq:neighboring farey forms}
\end{equation}
(See, for instance, \cite[Lemma 1]{MZ}.) Note that the first and last elements of any given $\mathcal{F}_m$ are always $\frac{0}{1}$ and $\frac{1}{1}$, which have CF expansions $\wedge$ and $[1]$ respectively. 

If $\frac{P}{Q}, \frac{P'}{Q'}\in \mathcal{F}_m$ are successive fractions and $Q>Q'$ (but we no longer make a demand of how $Q$ compares to the denominator of the other neighboring fraction), then we may still claim that $\frac{P}{Q}$ can be written as $[a_1,a_2,\dots, a_n+1]$ and $\frac{P'}{Q'}$ has one of the forms in \eqref{eq:neighboring farey forms}, simply by viewing both fractions as belonging to $\mathcal{F}_{Q}$ instead of $\mathcal{F}_m$, and here $\frac{P}{Q}, \frac{P'}{Q'}$ are still successive and $\frac{P}{Q}$ must have a denominator at least as large as both its neighbors.

The following fact is a direct consequence of this:
\begin{lem}\label{lem:compare point to Farey}Let $B$ be some block of CF digits with $q(B)\ge m$ for some $m\ge 2$. Suppose $P/Q, P'/Q'$ are consecutive fractions in $\mathcal{F}_m$ with
\[
\frac{P}{Q} < r_B < \frac{P'}{Q'}.
\] If $B_1$ is the shortest block such that $r_{B_1}=P/Q$ and if $B_2$ is the prefix of $B$ such that $r_{B_2}=r_B(m)$, then $-1\le |B_2|-|B_1|\le 4$ and the prefixes of $B_1$ and $B_2$ of length $|B_1|-1$ are the same.  
\end{lem}

\begin{proof}
If $Q>Q'$, then by the fact alluded to just before this lemma, we must be able to write $P/Q=[a_1,a_2,\dots, a_n+1]$ and $P'/Q'$ must be either 
\[
[a_1,a_2,\dots, a_n] \qquad \text{ or } \qquad [a_1,a_2,\dots, a_{n-1}].
\]
In the former case, the interval from $P/Q$ to $P'/Q'$ is, up to endpoints, equal to the cylinder set $C_{a_1a_2\dots a_n}$. In the latter case, the interval from $P/Q$ to $P'/Q'$ is, up to endpoints again, equal to the union of rank-$n$ cylinder sets
\[
\bigcup_{i=1}^{\infty} C_{a_1a_2 \dots (a_n+i)}.
\]
Thus, we see that $B$ must have at least $n=|B_1|$ digits and that the $n$th digit of $B$ is at least $a_n$ in size. Therefore, $|B_2|$ cannot be more than $|B_1|+2$ digits long, as otherwise, since the denominators of convergents at least double whenever two digits are appended and since $Q\ge m/2$, we would have $q(B_2)>m$, which contradicts the definition of $B_2$. On the other hand, we must have $|B_2|\ge |B_1|-1$, as $B$ must be prefixed by $a_1a_2\dots a_{n-1}$ and we know $q_{a_1a_2\dots a_{n-1}}\le Q\le m$. This proves the result in this case.   

The proof if $Q'>Q$ is similar, but now it could be that $P'/Q' = [a_1,a_2,\dots, a_n]$ and $P/Q = [a_1,a_2,\dots, a_{n-1}]$. If $a_{n-1}=1$, then $|B_1|=n-2$, giving the slightly worse upper bound of $|B_2|-|B_1|\le 4$.
\end{proof}

\begin{rem}\label{rem:Farey interval to CF interval}
Looking closer at the beginning of the previous proof, we notice that for any block $B$ such that $r_B\in \mathcal{F}_m$, the interval between $r_B$ and the succeeding Farey fraction is contained (up to endpoints) in the cylinder $C_{B_p}$, where $B_p$ is the prefix of $B$ of length $|B|-2$. The reason why we cannot make use of a longer prefix is because it may be that the interval has the form  $\bigcup_{i=1}^{\infty} C_{a_1a_2 \dots (a_n+i)}$ while $B$ ends on a $1$
\end{rem}

The following is a variant of Lemma 5 in \cite{AB}.

\begin{prop}\label{prop:total points in Fareys}
Suppose $m,d$ are positive integers with $m^2\le d$. Let $U$ denote any subset of $\mathcal{F}_m$, and for any $P/Q\in \mathcal{F}_m$, let $P'/Q'$ denote the consecutive element. Then
\[
\sum_{P/Q\in U} \#\left\{ a\in \mathbb{Z}_d^* \middle| \frac{P}{Q}< \frac{a}{d} < \frac{P'}{Q'} \right\} \le 2d\cdot \lambda\left( \bigcup_{P/Q\in U} \left\{ x\middle| \frac{P}{Q} \le x \le \frac{P'}{Q'}\right\}  \right),
\]
where $\mathbb{Z}_d^*$ denotes set of integers $[1,d]$ that are relatively prime to $d$.
\end{prop}

\begin{proof}We estimate as follows: 
\begin{align*}
\sum_{P/Q\in U} \#\left\{ a\in \mathbb{Z}_d^* \middle| \frac{P}{Q}< \frac{a}{d} < \frac{P'}{Q'} \right\}&\le \sum_{P/Q\in U}  \#\left\{ a\in [1,d] \middle| \frac{P}{Q}< \frac{a}{d} < \frac{P'}{Q'} \right\} \\
&\le \sum_{P/Q\in U} \left( d\left( \frac{P'}{Q'}-\frac{P}{Q} \right)+1\right)\\
&= d\cdot \lambda\left( \bigcup_{P/Q\in U} \left\{ x\middle| \frac{P}{Q} \le x \le \frac{P'}{Q'}\right\}  \right) +\#U.
\end{align*}
To estimate the size of $\# U$, note that
\begin{align}
    \# U &= \sum_{P/Q\in U} 1 \le\sum_{P/Q\in U} \frac{m^2}{QQ'} =  m^2 \sum_{P/Q \in U} \left(  \frac{P'}{Q'}-\frac{P}{Q} \right)\label{eq:U bound1}\\ & = m^2 \cdot \lambda\left( \bigcup_{P/Q\in U} \left\{ x\middle| \frac{P}{Q} \le x \le \frac{P'}{Q'}\right\}  \right)  , \label{eq:U bound2}
\end{align}
where in the inequality we made use of the fact that $Q,Q'\le m$, and in the last equality we made use of  \eqref{eq:consecutive farey denominators}. Since $m^2\le d$, this gives the desired bound.
\end{proof}

The following result of Avdeeva and Bykovskii \cite[Lemma 4]{AB} will be used as well:
\begin{lem}\label{lem:AB Lemma 4}
With the notation above we have
\[
\lambda\left( \bigcup_{P/Q\in U} \left\{ x\middle| \frac{P}{Q} \le x \le \frac{P'}{Q'}\right\}  \right) \ll \frac{1}{m} \sqrt{\# U},
\]
with a uniform implicit constant.
\end{lem}

The following result of Hensley \cite{Hensley} would normally be considered weaker than the result of Baladi and Vall\'{e}e \cite{BV}, but when one is so far from the expected value, it gives far better estimations.
\begin{lem}\label{lem:Hensley length bound}
There exists a constant $c_1$ such that for all sufficiently large $m$ and all $z\ge 0$ (independent of our choice of $m$), we have that
\[
\#\left\{ \frac{p}{q}\in \mathcal{F}_m: \left| L(p/q)-\lambda_{\text{KL}}^{-1} \log m\right| \ge z \sqrt{\log m}\right\} \le m^2 e^{-c_1 z^2}, 
\]
where $L(p/q)$ is the length of the shortest CF expansion of $p/q$ and $\lambda_{\text{KL}}$ is the Khinchin-L\'{e}vy constant $\frac{\pi^2}{12\log 2}$.
\end{lem}

From this we derive the following result immediately. We could prove a far stronger asymptotic, but the one here will suffice for our purposes.
\begin{lem}\label{lem:BV variant}
Let $N$ be an integer and $\epsilon>0$. Let \[
m:= \left\lfloor \exp\left((\lambda_{\text{KL}}+\epsilon) N\right)\right\rfloor.
\]
Then the proportion of elements in $\mathcal{F}_m$ that have less than $N$ digits in their shortest  CF expansion is at most
\[
O\left( \frac{1}{\sqrt{N}}\right)
\]
with a uniform implicit constant provided $N$ is sufficiently large (relative to a fixed $\epsilon$).
\end{lem}

The following related result is Proposition 2.4 in \cite{V16a}. 

\begin{lem}\label{lem:V bad denom bound}
For any $\epsilon>0$ and any positive integer $N$, we have that
\[
\mu\left(\left\{x\in[0,1)\setminus\mathbb{Q}: \left| \frac{\log q_N(x)}{N}-\lambda_{\text{KL}} \right| >\epsilon  \right\}  \right)= O\left(\frac{1}{N} \right),
\]
where the implicit constant is at most dependent on $\epsilon$.
\end{lem}

We will need this in the following form:
\begin{prop}\label{prop:bound on denominator normal}
Let $N$ be an integer and $\epsilon>0$. Let 
\[
m:= \left\lfloor \exp((\lambda_{\text{KL}}+\epsilon)N)\right\rfloor.
\]
Let $\mathcal{G}_m$ denote the subset of  $\mathcal{F}_m$ defined as follows: if $p/q\in \mathcal{F}_m$, $B$ is a block of continued fraction digits such that $r_B=p/q$, and there exists an integer $n\in [1,|B|-2]$ that is a multiple of  $\lfloor \sqrt{N}\rfloor$ such that
\begin{equation}
\left| \frac{\log q_n(B)}{n} - \lambda_{\text{KL}} \right| > \epsilon,\label{eq:prop on denominators inequality}
\end{equation}
then $p/q$ belongs to $\mathcal{G}_m$ as well.
Then
\[
\frac{\#\mathcal{G}_m}{\#\mathcal{F}_m} = O\left( \frac{\log N}{\sqrt{N}} \right).
\]
This result still holds even if the condition that $r_B=p/q$ is replaced with the condition that $r_{B_1B}=p/q$ for some block $B_1$, provided there is a uniform bound on the length of $B_1$.
\end{prop}

\begin{rem}
Since reversing the order of the continued fraction expansion will take the two-fold copy of $\mathcal{F}_m$ (corresponding to both possibilities of CF expansion for any given fraction) to a two-fold copy of  $\mathcal{F}_m$, the above proposition still holds with $p/q$ replaced by $p^*/q$ in \eqref{eq:prop on denominators inequality} 
\end{rem}

\begin{proof}[Proof of Proposition \ref{prop:bound on denominator normal}]
Let $n\in\mathbb{N}$ be a fixed  multiple of $\lfloor\sqrt{N}\rfloor$ and let $k$ be a fixed non-negative integer. Then consider how many blocks $B$ satisfy $r_{B_1B}\in \mathcal{F}_m$ for some $B_1$ with $|B_1|=k$, $n\le |B|-2$, and
\[
\left| \frac{\log q_n(B)}{n} - \lambda_{\text{KL}} \right| > \epsilon.
\]
By Remark \ref{rem:Farey interval to CF interval} and the assumption that $n\le |B|-2$, we see that the interval from $r_{B_1B}$ to the succeeding Farey fraction in $\mathcal{F}_m$ is completely contained in the rank-$n$ cylinder $C_{B_1B'}$, where $B'$ is the length-$n$ prefix of $B$. By arguing as in \eqref{eq:U bound1} and \eqref{eq:U bound2}, we see that the number of such blocks $B$ is at most a constant times $m^2$ times the Lebesgue measure of the corresponding Farey fraction intervals. These intervals are all contained in rank-$n$ cylinders $C_{B_1B'}$ where \eqref{eq:prop on denominators inequality} holds. If we let $\mathcal{B}$ denote the set of all such $B'$'s then note that
\[
\mu\left( \bigcup_{|B_1|=k} \bigcup_{B'\in\mathcal{B}} C_{B_1B'} \right) = \mu \left( T^{-k} \bigcup_{B'\in\mathcal{B}} C_{B'}\right) = \mu\left( \bigcup_{B'\in\mathcal{B}} C_{B'}\right).
\]Recalling that Lebesgue and Gauss measures differ by at most a multiplicative constant, we see that the Lebesgue measure of the intervals $C_{B_1B'}$ must be at most $O(n^{-1})$ by Lemma \ref{lem:V bad denom bound}. Thus there are at most $O(m^2/n)$ such blocks, and hence $O(m^2/n)$ corresponding fractions in $\mathcal{G}_m$.

The largest $n$ that could possibly be considered is at most $a\log m$ for some $a>0$. This is because the longest block would consist of just repeating $1$'s, and the denominator of $r_{(1)^k}$ is on the order of $\phi^{k}$ as seen above.

By summing over all relevant multiples of $\lfloor \sqrt{N}\rfloor$, we get that
\begin{align*}
    \frac{\#\mathcal{G}_m}{\#\mathcal{F}_m} &= O\left(  \frac{1}{\#\mathcal{F}_m}\sum_{\substack{n\le a \log m\\ \lfloor \sqrt{N}\rfloor | n }} \frac{m^2}{n}  \right)\\
    &= O\left( \frac{m^2}{\#\mathcal{F}_m\lfloor \sqrt{N}\rfloor} \sum_{k\le a\log m/\lfloor \sqrt{N}\rfloor} \frac{1}{k}\right)\\
    &= O\left( \frac{m^2}{\#\mathcal{F}_m\lfloor \sqrt{N}\rfloor} \cdot \log (a\log m/\lfloor \sqrt{N}\rfloor)\right),
\end{align*}
and  recalling that $\log m \asymp N$ and $\#\mathcal{F}_m\asymp m^2$, this is $O(\log N/ \sqrt{N})$. Finally we may sum over the possible values of $k$, which are possible lengths of $B_1$, but since we assumed this was uniformly bounded, it does not alter the $O(\log N/\sqrt{N})$ bound except in the constant.
\end{proof}

\section{$(\epsilon,u)$-normality and its variants}

Fix a digital system, either a base-$b$ system or the regular continued fraction expansion. Let $w$ be a finite or infinite word in our digital system and let $w^{(\ell)}=w_1w_2\dots w_\ell$ be the prefix consisting of the first $\ell$ digits (assuming that $\ell\le |w|$). Let $u$ be any finite word in our digital system. We will let 
\[
\nu_u(w^{(\ell)}) \text{ or, equivalently, } \nu_u(w,\ell)
\]
denote the number of appearances of $u$ in $w^{(\ell)}$. Note that for a finite word $w$, $\nu_u(w):=\nu_u(w,|w|)$. With the notation of the introduction, we have that $\nu_u(w(x),\ell) = N_u(x,\ell)$.

We will say that a finite word $w$ is $(\epsilon,u)$-normal (with respect to our digital system) for some $\epsilon>0$ if
\[
\left| \frac{\nu_u(w)}{|w|-|u|+1}-m(u)\right| \le \epsilon,
\]
where $m(u)$ is the expected frequency of the word $u$ in whatever system we are considering. This will be $b^{-|u|}$ in base-$b$ and will be $\mu(C_u)$ in the continued fraction system. In other words, $w$ is $(\epsilon,u)$-normal if the frequency which $u$ appears in $w$ is within $\epsilon$ of the desired frequency $m(u)$.

Importantly, in this paper, if $|w|<|u|$, then a word $w$ is automatically considered to \emph{not} be $(\epsilon,u)$-normal.

Let $\mathcal{U}$ be a collection of words. We will say that a word $w$ is $(\epsilon,\mathcal{U})$-normal if it is $(\epsilon,u)$-normal for each $u\in \mathcal{U}$. We will say that a word $w$ is $(\epsilon,k)$-normal if it is $(\epsilon,u)$-normal for every word $u$ with $|u|=k$.

We will say that a string is $(\epsilon,u,m)$-normal if it is $(\epsilon,u)$-normal for each prefix of $u$ whose length is a multiple of $m$. We can extend this definition likewise to $(\epsilon,\mathcal{U},m)$-normal and $(\epsilon,k,m)$-normal strings. In a base-$b$ system, we will say that a string $s$ is $(\epsilon,k,m)^*$-normal if both $s$ and $s^*$ are $(\epsilon,k,m)$-normal.

We make use of an additional definition for continued fractions. We say a block $B$ is $(\epsilon,m)$-denominator normal if for every prefix $B'$ of $B$ whose length is a multiple of $m$, we have
\[
\left| \frac{\log q(B')}{m}-\lambda_{\text{KL}} \right| < \epsilon.
\]

\subsection{Continued fraction normality}

The following result, which bounds the Gauss measure of the set of points whose first $n$ CF digits form a non-$(\epsilon,A)$-normal block, is due to Adrian-Maria Scheerer \cite{Scheerer}. 
\begin{lem}\label{lem:Scheerer}
Let $\epsilon>0$ and fix a block $A$ of $\ell$ positive integers. There is a constant $\eta_{CF}(\epsilon,A)>0$ such that for $n\ge 2(\ell+1)$, we have that
\[
\mu(E_{CF}^c(\epsilon,A;n)) \le \exp\left( -\eta_{CF}(\epsilon,A) \frac{n}{\log n}\right),
\]
where $\mu$ is the Gauss measure, and $E_{CF}(\epsilon,A;n)$ is the set of $x\in [0,1)$ whose first $n$ CF digits form an $(\epsilon,A)$-normal block.
\end{lem}

We need a slightly refined version of Scheerer's estimate, which bounds the Gauss measure of the set of points whose first $Kn$ CF digits form a non-$(\epsilon,A,n)$-normal block.

\begin{lem}\label{lem:Scheerer-extended}
Let $\epsilon>0$ and let $\mathcal{A}$ be a finite collection of blocks. There exist constants $\eta=\eta(\epsilon,\mathcal{A})>0$ and $\xi=\xi(\epsilon,\mathcal{A})>0$, such that for $n$ satisfying
\[
n \ge \max\left\{ 2\left( \max_{A\in \mathcal{A}} |A|+1\right), 6, \frac{3}{\eta}\right\},
\] we have that, for any integer $K\ge 1$ or $K=\infty$, \[
\mu\left(\bigcup_{k=1}^K E_{CF}^c(\epsilon,\mathcal{A};kn)\right) \le \xi \exp\left( -\eta \frac{n}{\log n}\right),
\]
where $E_{CF}(\epsilon,\mathcal{A};kn)$ is the set of $x\in[0,1)$ whose first $n$ CF digits form an $(\epsilon,\mathcal{A})$-block.
\end{lem}

\begin{proof}
It is clear from Lemma \ref{lem:Scheerer} that
\[
\mu(E_{CF}^c(\epsilon,\mathcal{A};m))\le\sum_{A\in \mathcal{A}} \mu(E_{CF}^c(\epsilon,A;m)) \le  |\mathcal{A}| \exp \left( -\eta \frac{m}{\log m} \right),
\]
where $\eta=\min_{A\in \mathcal{A}} \eta_{CF}(\epsilon,A)$ and $m\ge 2(\max_{A\in\mathcal{A}}|A|+1)$. Consequently, we have the following:
\begin{align*}
    \mu\left( \bigcup_{k=1}^\infty E_{CF}^c(\epsilon,\mathcal{A};kn)\right) &\le \sum_{k=1}^\infty \mu\left( E_{CF}^c(\epsilon,\mathcal{A};kn)\right)\\
    &\le \sum_{k=1}^\infty |\mathcal{A}| \exp\left( -\eta \frac{kn}{\log kn} \right)
\end{align*}

We may estimate this sum using the integral test. Notably, the function
\[
f(x) = e^{-\eta \frac{xn}{\log xn}}
\]
has a negative derivative for $x>e/n$. Since we have assumed that $n\ge 6$, this is true for all $x\ge 1$. Therefore, we have that
\[
    \mu\left( \bigcup_{k=1}^\infty E_{CF}^c(\epsilon,\mathcal{A};kn)\right) \le |\mathcal{A}| \exp\left(-\eta \frac{n}{\log n}\right) +|\mathcal{A}|\int_1^\infty \exp\left(-\eta \frac{xn}{\log (xn)}\right) \ dx.
\]

Next, by our assumption on the size of $n$, it is easy to check that $\eta n \ge 3$ and $\log^2 n\ge 3$. Therefore,
\[
3\left( 1- \frac{1}{\eta n x} - \frac{1}{\log^2(xn)}\right) \ge 1
\]
for $x\ge 1$.

In particular, we have that
\begin{align*}
    &\mu\left( \bigcup_{k=1}^\infty E_{CF}^c(\epsilon,\mathcal{A};kn)\right)\\ &\qquad\le |\mathcal{A}| \exp\left(-\eta \frac{n}{\log n}\right)\\ &\qquad\qquad +|\mathcal{A}|\int_1^\infty \exp\left(-\eta \frac{xn}{\log (xn)}\right)\cdot 3\left( 1- \frac{1}{\eta n x} - \frac{1}{\log^2(xn)}\right) \ dx\\ 
    &\qquad=  |\mathcal{A}| \exp\left(-\eta \frac{n}{\log n}\right) \\ &\qquad\qquad +3|\mathcal{A}|\int_1^\infty \frac{d}{dx} \left( \frac{\exp\left( -\eta\frac{xn}{\log(xn)}\right)( \log(xn)+1)}{-\eta n}\right) \ dx\\ 
    &\qquad= |\mathcal{A}| \exp\left(-\eta \frac{n}{\log n}\right) +3|\mathcal{A}|\frac{\exp\left( -\eta \frac{n}{\log n}\right) (\log n + 1)}{\eta n}\\
    &\qquad\le \xi \exp \left( -\eta \frac{n}{\log n}\right),
\end{align*}
where $\xi = |\mathcal{A}|\left( 1+\frac{3}{\eta}\right)$. In the final line, we used that $n\ge \log n + 1$ for all $n>0$. This proves the desired relation. \end{proof}

We also need the following, more intricate result.

\begin{lem}\label{lem:Scheerer reversed}
Let $\epsilon>0$ and $\mathcal{A}$ be a collection of blocks. Then under the same assumptions as Lemma \ref{lem:Scheerer-extended}, we have that, for any $K\ge n$,  \[
\mu\left( E^c_{CF}(\epsilon,\mathcal{A},n;K)^*\right) \le K \xi \exp\left( -\eta \frac{n}{\log n}\right),
\]
where $E^c_{CF}(\epsilon,\mathcal{A},n;K)^*$ is the union of all cylinder sets $C_B$ where $n\le |B|\le K$ and $B^*$ is \emph{not} $(\epsilon,\mathcal{A},n)$-normal.
\end{lem}

\begin{proof}
Let $\mathcal{B}_k$ denote the set of all blocks $B$ with $|B|=k$ and $B^*$ not $(\epsilon,\mathcal{A},n)$-normal, with $n\le k \le K$. Since the cylinder sets corresponding to different blocks $B$ in $\mathcal{B}_k$ are disjoint, by Lemma \ref{lem:C_B versus C_B^*} we have that
\[
\mu\left( \bigcup_{B\in\mathcal{B}_k} C_B\right) = \mu\left(\bigcup_{B\in\mathcal{B}_k} C_{B^*}\right).
\]
By applying Lemma \ref{lem:Scheerer-extended}, we see that the latter is at most $\xi \exp(-\eta n/\log n)$. Summing over all possible values for $|B|$ gives the desired result. 
\end{proof}

\subsection{base-$b$ normality}

We will make use the following result, in the form seen in Pollack and Vandehey \cite[Proposition 2]{PV}, although it derives from an earlier result of Copeland and Erd\H{o}s \cite{CE}.

\begin{lem}\label{lem:Copeland-Erdos}
Fix a base $b$. Let $\epsilon>0$ and $k\in\mathbb{N}$ be fixed. There exists $\delta=\delta(\epsilon,k,b)$, such that the number of base-$b$ strings of length $\ell$ that are not $(\epsilon,k)$-normal to base-$b$ is at most $b^{\ell(1-\delta)}$ for all sufficiently large $\ell$.
\end{lem}

We need the following variant of the above result.

\begin{lem}\label{lem:Copeland-Erdos var}
Fix a base $b$. Let $\epsilon>0$ and $k\in \mathbb{N}$ be fixed. Then the number of base-$b$ strings of length $\ell$ that are not $(\epsilon,k,m)$-normal (or $(\epsilon,k,m)^*$-normal) to base-$b$ is at most
\[
O(b^{\ell-\delta m}),
\]
where $\delta=\delta(\epsilon,k,b)$ is as in Lemma \ref{lem:Copeland-Erdos}, provided $m$ is large enough. The implicit constant in the big-O may be different depending on whether we consider $(\epsilon,k,m)$-normality or $(\epsilon,k,m)^*$-normality, but is otherwise independent of other variables.
\end{lem}

\begin{proof}The number of strings of length $\ell$ whose first $im$ digits ($1\le im \le \ell$) do not form a $(\epsilon,k)$-normal number is at most $b^{im(1-\delta)+\ell-im}$ by Lemma \ref{lem:Copeland-Erdos}. So the number of strings of length $\ell$ counted by the lemma is at most
\[
\sum_{1\le im\le \ell} b^{im(1-\delta)+\ell-im}=b^\ell \sum_{1\le im \le \ell} b^{-\delta im}  \le b^{\ell} \frac{b^{-\delta m }}{1-b^{-\delta m}}.\]
By choosing $m$ large enough so that $1-b^{-\delta m}\ge 1/2$, we get the desired result for non-$(\epsilon,k,m)$-normal strings. 

Finally, we see that the number of non-$(\epsilon,k,m)^*$-normal strings is at most twice the number of non-$(\epsilon,k,m)$-normal strings, since it is at most the number of strings $s$ that are non-$(\epsilon,k,m)$-normal plus the number of strings $s$ such that $s^*$ is non-$(\epsilon,k,m)$-normal, and the operation $s\mapsto s^*$ is a bijection on the strings of a fixed length.
\end{proof}

We will need to apply the previous result in the context of the following result. Fundamentally what the next lemma says is that most numbers that are really close to a $b$-adic rational with $b$ odd will tend to have really well-behaved binary expansions.

\begin{lem}\label{lem:Copeland-Erdos fraction variant}
Let $b\ge 3$ be an odd prime  integer base. Let $j\ge 1$ be an integer. Let $M=b^{\ell j}$ for some positive integer $\ell$.

Let $I\subset [0,1)$ be an interval and $j_0=-\lceil \log_{2}\lambda(I) \rceil $, and define $J=\lceil j \log_2 b  \rceil- j_0$. We will assume that $j$ is large enough that $J\ge \lceil j \log_2 b \rceil/2$. Let $\epsilon>0$ and $k$, $m$ be integers.

For each block $B$ such that $r_B\in I$ and $q(B)=b^j$, let $s=S_2(C_{BM})$; and decompose $s$ as the concatenation of $c_0c_1\dots c_{t}$ such that $|c_0|=j_0$,  $|c_1|=|c_2|=\dots  = |c_{t-1}|=J$, and $|c_t|\le J$. Then there are at most \[O\left(\ell \lambda(I) b^{j}2^{-\delta m}\right)\] different blocks $B$ where at least one of the strings $c_1, \dots, c_{t-1}$ is not $(\epsilon,k,m)^*$-normal in base 2, provided $m$ is sufficiently large (dependent on $\epsilon,k$), with $\delta=\delta(\epsilon,k,2)$ as in Lemma \ref{lem:Copeland-Erdos var}.
\end{lem}

The idea of breaking $S_2(C_{BM})$ into several pieces and analyzing the normality properties of them separately comes from the main result of \cite{Besicovitch}, see also \cite{PV}.

\begin{proof}
Let us fix a block $B$. Then the cylinder set $C_{BM}$ is the interval between the points $r_{BM}=\frac{p(BM)}{q(BM)}$ and $r_{BM1}=\frac{p(BM)+p(B)}{q(BM)+q(B)}$. Let $a$ be the integer relatively prime to $b$ such that $r_B=a/b^j$. Then 
\[
r_{BM} = \frac{a}{b^j} + \frac{(-1)^{|B|+1}}{b^j q(BM)} \qquad r_{BM1} = \frac{a}{b^j}+ \frac{(-1)^{|B|+1}}{b^j (q(BM)+b^j)} .
\]
Note moreover that $q(BM)$ is relatively prime to $b$. Using Proposition \ref{prop:product of denominators}, we see that 
\[
    L_2(C_{BM}) = 2\log_2 q(BM) +O(1)= 2\log_2 (q(B) \cdot M)+ O(1)=2(\ell+1)j\log_2 b+O(1).
\]
Since $j_0+(t-1)J\le L_2(C_{BM})$, this implies that $t\ll \ell j \log_2(b)/J$. By our assumption on the size of $J$, this in turn gives $t \ll \ell$.

The point $r_{BM}$ is to the left of $r_{BM1}$ if and only if  $|B|$ is odd. Since every fraction of the form $a/b^j$ has two corresponding ways of writing its continued fraction expansion one whose length is even and one whose length is odd, we will for the moment assume that $|B|$ is always even and instead of counting blocks, we will instead count $a$'s. For any given $a$ we will let $B(a)$ denote the corresponding block of even length.

First, consider how many possible $a$'s can have $c_1$ be non-$(\epsilon,k,m)^*$-normal. The only way for two values $a,a'$ to give rise to the same $c_0c_1$ is if 
\begin{equation}
\frac{1}{2^{|c_0c_1|}} \ge \left| r_{B(a)M}-r_{B(a')M}\right| \ge \frac{|a-a'|}{b^j} - \left| \frac{1}{b^j q(B(a)M)}- \frac{1}{b^jq(B(a')M)}\right|,\label{eq:copelanderdos frac var 1}
\end{equation}
where in the last inequality, we applied our explicit formulas for $r_{BM}$ above. By construction of $J$, we see the the left-hand side is at most $b^{-j}$. Moreover, the absolute value on the far right-hand side is smaller than $b^{-\ell j}$ and so it is clear that at most 2 different values of $a$ can give rise to the same $c_0c_1$. However, as noted in Remark \ref{rem:base-b digit variants}, there could be as many as $3$ separate values for $c_0$ for different fractions in $I$. Therefore, any given $c_1$ can appear for no more than $6$ different values of $a$. Since, by Lemma \ref{lem:Copeland-Erdos var}, we have that at most 
$ O(2^{J-\delta m})$ of the $c_1$'s are not $(\epsilon,k,m)^*$-normal, we get that there are likewise   at most $ O(2^{J-\delta m})$ values of $a$ which make $c_1$ not $(\epsilon,k,m)^*$-normal.

Now we consider for how many $a$'s give rise to $c_i$, for a fixed value of $i\in [2,t-1]$, that are not $(\epsilon,k,m)^*$-normal. Suppose $a,a'$ give rise to the same $c_i$, and without generality assume that $a>a'$. Then the definition of $c_i$ implies that
\[
\lfloor 2^{j_0+iJ} r_{B(a)M}\rfloor \equiv \lfloor 2^{j_0+iJ} r_{B(a')M}\rfloor \pmod{2^J} .
\]
We can approximate $r_{BM}$ very well. Namely,
\begin{align*}
    r_{BM} &= \frac{a}{b^j} -\frac{1}{b^j q(BM)}\\
     &= \frac{a}{b^j} -\frac{1}{b^j (Mb^j + O(b^j))}\\
      &= \frac{a}{b^j} -\frac{1}{Mb^{2j}}+O\left(\frac{1}{M^2b^{2j}} \right).
\end{align*}
With this approximation, the above modular equivalence implies that there is some integer $n$ such that
\begin{equation}
\frac{2^{j_0+iJ}}{b^j} (a-a') = n 2^J +O(1)+ O\left( \frac{2^{j_0+iJ}}{M^2b^{2j}}   \right). \label{eq:copelanderdos frac var 2}
\end{equation}
The $O(1)$ comes from replacing $\lfloor z \rfloor$ with $z$. Since $j_0+iJ\le |c_0c_1\dots c_t|= L_2(C_{BM})$, which we showed earlier was at most $2(\ell+1)j\log_2 b+O(1)$, and since $M=b^{\ell j}$, the final big-O term in the previous displayed equation is $O(1)$ as well. Moreover since $(a-a')/b^j\in [0,\lambda(I)]$, we have that $n$ is within $O(1)$ of $[0, 2^{j_0+(i-1)J}\lambda(I)]$.

We now rearrange the above equation to get
\[
2^{j_0+(i-1)J}(a-a')=b^j n + O\left( \frac{b^j}{2^{J}}\right).
\]
The big-O term on the right-hand side is at most $O(1)$. So if we take this equation modulo $2^{(i-1)J}$, we get
\[
b^j n +O(1) \equiv 0 \pmod{2^{(i-1)J}}.
\]
First, note that $2^{j_0+(i-1)J} \lambda(I) \ll 2^{(i-1)J}$, so if we allow $n$ to vary over its range, then $n$ runs through at most a uniformly bounded number of complete residue sets. Since $b$ and $2$ are assumed to be relatively prime, this means that $b^j n$ runs through at most a uniformly bounded number of complete residue sets. And therefore, there are at most $O(1)$ solutions to the above equivalence. And this in turn implies that at most $O(1)$ values of $a$ can give rise to the same $c_i$. Since, by Lemma \ref{lem:Copeland-Erdos var}, we have that at most $ O(2^{J-\delta m})$ of the $c_i$'s are not $(\epsilon,k,m)^*$-normal, and hence at most $O(2^{J-\delta m})$ of the $a$'s give rise to $c_i$'s that are not $(\epsilon,k,m)^*$-normal. 

Summing up over all $i$'s from $1$ to $t-1$, we see that there are at most $O(t 2^{J-\delta m})$ values of $a$ which give rise to at least one of the strings $c_1, \dots, c_{t-1}$ being not $(\epsilon,k,m)^*$-normal. 

The proof when we assume $|B|$ is odd gives the same bound. Recalling our bound on $t$ from above and the definition of $J$, gives the desired result.
\end{proof}

The following variant of the above lemma will also be helpful: 
\begin{lem}\label{lem:Copeland-Erdos fraction variant 2}
Let $b,b'\ge 2$ be co-prime  integer bases. Let $j\ge 1$ be an integer.

Let $I\subset [0,1)$ be an interval and $j_0=-\lceil \log_{b'}\lambda(I) \rceil $, and define $J=\lceil j \log_{b'} b  \rceil- j_0$. We will assume that $j$ is large enough that $J\ge \lceil j \log_{b'} b \rceil/2$. Let $\epsilon>0$ and let $k,m$ be integers.

For each block $B$ such that $r_B\in I$ and $q(B)=b^j$, let $s=S_{b'}(C_{B})$; and decompose $s$ as the concatenation of $c_0c_1c_2c_3$ such that $|c_0|=j_0$,  $|c_1|=|c_2|=J$, and $|c_3|= j_0+O(1)$. Then there are at most \[O\left( j \lambda(I) b^{j}{b'}^{-\delta m}\right)\] different blocks $B$ where at least one of the strings $c_1, c_2$ is not $(\epsilon,k,m)^*$-normal in base $b'$, provided $m$ is sufficiently large in terms of $\epsilon$ and $k$, with $\delta=\delta(\epsilon,k,b')$ as in Lemma \ref{lem:Copeland-Erdos var}.
\end{lem}

\begin{proof}
Most of the details follow in a similar manner to the previous lemma with the role of $2$ replaced with $b'$. One notable exception is that instead of considering $r_{BM}$ and $r_{BM1}$, we consider
\[
r_{B} = \frac{a}{b^j} \qquad \text{and} \qquad r_{B1} = \frac{a}{b^j} + \frac{(-1)^{|B|+1}}{b^j q(B1)}.
\]
Both of these points can be represented as $a/b^j+O(1/b^{2j})$. In the previous proof, this means the last term of \eqref{eq:copelanderdos frac var 1} would be replaced with a term of size $O(1/b^{2j})$ and the last term of \eqref{eq:copelanderdos frac var 2} would be of size $O({b'}^{j_0+2J}/b^{2j})$. But this results in no substantial change to the remainder of the proof. 

We also need to prove the refined bound on the size of $|c_3|$. For this, note that $L_2(C_B) = 2\log_{b'} b^j+O(1)$ and that $J=\lceil j \log_{b'} b\rceil-j_0$. So \[ |c_3| = L_2(C_B)-|c_0c_1c_2| = j_0+O(1),\] 
as desired.
\end{proof}

Our reason for studying $(\epsilon,k,m)^*$-normality rather than $(\epsilon,k,m)$-normality is given in the following result.

\begin{lem}\label{lem: adding does not alter normality}
Suppose $s$ is an $(\epsilon,k,m)^*$-normal binary string of length with $m-k\ge 2$ and $\epsilon<1/6$. Then $s+1$ and $s+2$, interpretted as binary strings of the same length as $s$, can be decomposed as an $(\epsilon,k,m)$-normal string and a string of length at most $m+2$.
\end{lem}

\begin{proof}
If $s$ decomposes as $s'00$, then $s+1=s'01$ and $s+2=s'10$. In this case, the proof is trivial.

Otherwise, either $s$ decomposes as $s'0(1)^j$, with $s+1=s'1(0)^j$ and $s+2= s'1(0)^{j-1}1$, or $s$ decomposes as $s'0(1)^j0$, then $s+1= s'0(1)^{j+1}$ and $s+2 = s'1(0)^{j+1}$. In each of these two cases, we want to consider how large $j$ could be. Suppose $j\ge m$. By our assumption that $s$ is $(\epsilon,k,m)^*$-normal, we know that the string of the first $m$ digits of $s^*$, which will either be $(1)^m$ or $0(1)^{m-1}$, is $(\epsilon,k)$-normal. However, $(1)^k$ occurs with frequency at least \[
\frac{m-k}{m-k+1}  \ge  \frac{2}{3},
\]
which exceeds 
\[
m((1)^k)+\epsilon< \frac{1}{2}+ \frac{1}{6} = \frac{2}{3}.
\]
Thus $j$ must be strictly less than $m$.
\end{proof}

\subsection{$(\epsilon,u)$-normality and concatenating words}

In this section we will examine how concatening words affects the normality properties of the new words. This will lead us to a general rule which we will use to prove the normality or non-normality of infinite words.

\begin{lem}\label{lem:epsilon-u subwords}
Suppose $w_1,w_2,\dots, w_k$ are all finite $(\epsilon,u)$-normal words of length at least $|u|$, with $\epsilon\le1$. Then any word $v$ which contains $w_1,w_2,\dots, w_k$ as disjoint subwords, is
\[
\left( \epsilon+\frac{ (|u|+2)\left( |v|-\sum_{i=1}^k |w_i|+4k+2\right)}{|v|-|u|-1},u \right)-\text{normal}.
\]
\end{lem}

Note that for the above lemma to be truly useful, we need to have that $v$ is not much longer than the combined length of the $w_i$'s.

\begin{proof}
We can rewrite the condition that $w_i$ is $(\epsilon,u)$-normal as $\nu_u(w_i) = m(u) (|w_i|-|u|+1)+O(\epsilon(|w_i|-|u|+1))$, with an implicit constant of $1$. Suppose $v$ satisfies the conditions of the lemma. First, note that
\[
\nu_u(v) = \sum_{i=1}^k \nu_u(w_i)+O(2k(|u|-1)) + O\left(|u|\left(|v|-\sum_{i=1}^k |w_i| \right)\right).
\]
The first term accounts for all occurrences of $u$ in $v$ which occur entirely within a single $w_i$. The second term bounds the occurrences of $u$ which occur partially but not entirely within a single $w_i$. There are at most $2(|u|-1)$ positions for a string $u$ to contain elements both inside and outside $w_i$, and there are $k$ such words $w_i$. The third term accounts for any occurrence of $u$ which contains one of the  $|v|-\sum_{i=1}^k |w_i|$  elements of $v$ that do not occur in any $w_i$. As any such element could appear in $|u|$ such strings, we see that the above is true with implicit constant $1$.

Hence, we have that
\begin{align*}
 \nu_u(v) &= \sum_{i=1}^k\left[ \left( m(u) (|w_i|-|u|+1)+O(\epsilon(|w_i|-|u|+1)) \right)\right]+O(2k(|u|-1))\\ &\qquad + O\left(|u|\left(|v|-\sum_{i=1}^k |w_i| \right)\right)\\
 &= m(u) \sum_{i=1}^k |w_i| + O(\epsilon\sum_{i=1}^k |w_i|) +O(4k(|u|-1))\\ &\qquad + O\left(|u|\left(|v|-\sum_{i=1}^k |w_i| \right)\right)\\
 &= m(u) |v|+O(\epsilon |v|) +O(4k(|u|-1)) \\ &\qquad+O\left( (|u|+2)\left( |v|-\sum_{i=1}^k |w_i|\right)\right)\\
 &= m(u)(|v|-|u|+1)+ O(\epsilon(|v|-|u|+1)) + O(2(2k+1)(|u|-1))\\ &\qquad +O\left( (|u|+2)\left( |v|-\sum_{i=1}^k |w_i|\right)\right)\\
  &= m(u)(|v|-|u|+1)+ O(\epsilon(|v|-|u|+1)) \\ &\qquad+ O\left( (|u|+2)\left( |v|-\sum_{i=1}^k |w_i|+4k+2\right)\right).
\end{align*}
This proves the result.
\end{proof}

\begin{prop}\label{prop:normality test}
Suppose that $(w_i)_{i\in\mathbb{N}}$, $(v_i)_{i\in\mathbb{N}}$ are sequences of finite words and consider the infinite alternating concatenation
\[
W= w_1v_1w_2v_2w_3v_3\dots.
\]

If there exists an $\epsilon>0$ and a digit $d$ in our digit set such that for infinitely many different $i$'s, we have that the only digit in $v_i$ is $d$ and that $|v_i|>\epsilon |w_1v_1\dots w_{i-1}v_{i-1}w_i|$, then $W$ is not normal.

If, on the other hand, the following conditions are satisfied:
\begin{itemize}
    \item As $i\to \infty$, we have that \begin{equation*}
        |v_1v_2\dots v_i| = o\left( |w_1v_1\dots w_{i-1}v_{i-1}w_i| \right) . 
    \end{equation*}
    \item There exists a sequence $(\epsilon_i)_{i\in \mathbb{N}}$ of positive reals tending to $0$, a sequence $(\mathcal{U}_i)_{i\in\mathbb{N}}$ of sets of words such that $\mathcal{U}_i\subset \mathcal{U}_{i+1}$ and such that eventually every finite word (in our system) appears in some $\mathcal{U}_i$, and a sequence of positive integers $(m_i)_{i\in \mathbb{N}}$ tending to infinity such that
    \begin{equation}
    \lim_{i\to \infty}  \frac{\sum_{j\le i} m_j}{|w_1\dots w_{i-1}|} =0 \label{eq:normality test bound}
    \end{equation}
    and, finally, such that each $w_i$ is $(\epsilon_i,\mathcal{U}_i,m_i)$-normal.
\end{itemize}
Then $W$ is normal.
\end{prop}

\begin{proof}
For the first part of the proposition, we note that regardless of which system we are in, $m((d)^k)$ decays exponentially to $0$ as $k$ tends to infinity. For any $i\in\mathbb{N}$, let $N_i=|w_1v_1\dots w_iv_i|$. Then for an infinite set of $i$'s we have that 
\[
\frac{\nu_{(d)^k}(W^{(N_i)})}{N_i} \ge \frac{|v_i|-k+1}{N_i}\ge \frac{\epsilon}{1+\epsilon}+o(1),
\]
where $o(1)$ is tending to $0$ as $i$ tends to infinity. By choosing $k$ sufficiently large so that $m((d)^k)$ is much smaller than $\epsilon/(1+\epsilon)$, we obtain that $W$ cannot be normal.

For the second part of the proposition, note first that if \[|v_1v_2\dots v_i| = o\left( |w_1v_1\dots w_{i-1}v_{i-1}w_i| \right) ,\] then all the digits in $W$ that come from the $v_i$'s have asymptotic density $0$, and therefore, the normality (or non-normality) of $W$ is unchanged if we assume that the $v_i$'s are all empty words, which we will now assume. For any finite word $u$ and any $\epsilon>0$, there is some $i_0$ such that $u$ appears in all $\mathcal{U}_i$ with $i\ge i_0$, $\epsilon_i\le \epsilon$ for $i\ge i_0$, and $|u|\le |m_i|$ for $i\ge i_0$. Let $n$ be an especially large positive integer and suppose that the $n$th digit of $W$ appears in the word $w_i$ (with $i\ge i_0$). Let $\ell_i m_i$ be the largest integer multiple of $m_i$ such that $w_1\dots w_{i-1}w_i^{(\ell_i m_i)}$ is a prefix of $W^{(n)}$. And for each $j\in [i_0,i)$, let $\ell_j m_j$ be the largest multiple of $m_j$ that is at most the length of $w_j$. Then we may think of $W^{(n)}$ as having $w_j^{(\ell_j m_j)}$, for $j\in [i_0,i]$ as subwords and apply Lemma \ref{lem:epsilon-u subwords}. This tells us that $W^{(n)}$ is $(\epsilon',u)$-normal, where
\[
\epsilon'=\epsilon+\frac{(|u|+2)\left(n-\sum_{j\in[i_0,i]} \ell_jm_j+4(i-i_0+1)+2\right)}{n-|u|-1} .
\]
Let us bound this crudely. We see that
\[
\epsilon'\le \epsilon+ \frac{(|u|+2)\left(|w_1w_2\dots w_{i_0-1}|+\sum_{j\le i}  m_j+4i+6\right)}{|w_1w_2\dots w_{i-1}|-|u|-1}.
\]
However, by our assumptions on the $w_i$'s this latter term must tend to zero as $n$ (and hence $i$) increases. Thus we see that the limiting density of $u$ in $W$ must be $m(u)+O(\epsilon)$. Since $\epsilon$ can be taken arbitrarily small, and since $u$ is an arbitrary word, we see that $W$ is normal in this case.
\end{proof}

The above proposition must be altered if our definition of normality is not $(\epsilon,\mathcal{U})$-normality, but rather $(\epsilon,k)$-normality for base-$b$ normality.

\begin{prop}\label{prop:normality test 2}
Suppose that $(w_i)_{i\in\mathbb{N}}$, $(v_i)_{i\in\mathbb{N}}$ are sequences of finite words in a base-$b$ expansion\footnote{We continue with the notation of the previous proposition even though these should technically be called strings.} and consider the infinite alternating concatenation
\[
W= w_1v_1w_2v_2w_3v_3\dots.
\]

If the following conditions are satisfied:
\begin{itemize}
    \item As $i\to \infty$, we have that \begin{equation}
        |v_1v_2\dots v_i| = o\left( |w_1v_1\dots w_{i-1}v_{i-1}w_i| \right) . \label{eq:normality test bound v's}
    \end{equation}
    \item There exists a sequence $(\epsilon_i)_{i\in \mathbb{N}}$ of positive reals tending to $0$, a sequence $(k_i)_{i\in\mathbb{N}}$ of positive integers tending to infinity, and a sequence of positive integers $(m_i)_{i\in \mathbb{N}}$ tending to infinity such that
    \begin{equation*}
    \lim_{i\to \infty}  \frac{\sum_{j\le i} m_j}{|w_1\dots w_{i-1}|} =0,
    \end{equation*}
    and
        \begin{equation}
    \lim_{i\to \infty}    \frac{2k_i}{m_i-k_i+1}+ b^{k_i} \epsilon_i =0, \label{eq:normality test bound 2}
    \end{equation}
    and, finally, such that each $w_i$ is $(\epsilon_i,k_i,m_i)$-normal.
\end{itemize}
Then $W$ is normal.
\end{prop}

\begin{proof}
The proof is nearly identical to that of the second part of Proposition \ref{prop:normality test}. The only issue with that proof is that any given word $u$ appears in all $\mathcal{U}_i$ once $i$ is big enough, whereas a given $u$ has a fixed length and so is only directly addressed by $(\epsilon_i,k_i,m_i)$-normality for finitely many $i$'s. To remedy this we need a rule which tells us how $(\epsilon,k)$-normality implies $(\epsilon',\ell)$-normality for $\ell< k$.

Suppose a word $w$ is $(\epsilon,k)$-normal, and let $u$ be a word of length $\ell$ which is strictly less than $k$. Then we have
\begin{align*}
    \nu_u(w) &= \sum_{|v|=k-\ell} \nu_{uv}(w) +O(k)\\
    &= \sum_{|v|=k-\ell}  \left( m(uv)(|w|-|uv|+1)+O(\epsilon(|w|-|uv|+1))  \right)+O(k)\\
    &= m(u)(|w|-k+1)+O(b^{k-\ell}\epsilon (|w|-k+1)) +O(k)\\
    &= m(u)(|w|-|u|+1)-m(u)(k-\ell)+O(b^{k}\epsilon (|w|-k+1)) +O(k)\\
    &= m(u)(|w|-|u|+1)+O(b^{k}\epsilon (|w|-k+1)) +O(2k),
\end{align*}
where all the big-O constants are at most $1$. The $O(k)$ comes from the fact that the appearances of $u$ in the last $k$ places of $w$ might not be counted in the sum. So in particular, the word $w$ is also
\[
\left( \frac{2k}{|w|-k+1}+ b^k \epsilon , \ell\right)\text{-normal.}
\]
Combining this with \eqref{eq:normality test bound 2} completes the necessary changes to the proof.
\end{proof}

\section{Proof of Theorem \ref{thm:main}}

\subsection{Initial set-up} We will follow the basic ideas outlined in Section \ref{sec:outline}. Let $z\in \omega^\omega$ be arbitrary. We will construct a continuous map $\phi$ from $\omega^\omega$ to $(0,1)$ such that $\phi(z)$ is CF-normal and absolutely abnormal if and only if $z\in C\setminus D$, i.e., if and only if $z$ tends to infinity on its odd indices but returns to some number infinitely often on its even indices. We will construct $\phi(z)$ based on its CF expansion so that the prefix of $z$ will determine the prefix of the CF expansion of $\phi(z)$ to guarantee that $\phi$ is continuous.

Let $\Delta=\{(i,j):i\in\mathbb{N}, 1\le j \le i\}\subset \mathbb{N}^2$. For $(i,j)\in\Delta$, let $b_{i,j}$ be the $j$th prime number. We will think of $(b_{i,j})_{(i,j)\in\Delta}$ as a sequence ordered  by the lexicographical order on $\Delta$. In particular, if $(i,j)\in\Delta$, then, abusing notation, we will let $(i,j+1)$ denote the successive element under the lexicographical order. So, if $j=i$, then $(i,j+1)=(i+1,1)$. Likewise $(i,j-1)$ will denote the preceding element.

We will let $(\epsilon_{i,j})_{(i,j)\in\Delta}$ be a sequence of positive real numbers tending to zero along $\Delta$, all sufficiently small so that $\epsilon_{i,j}\le \lambda_{\text{KL}}/2$ and satisfying that for any $C>1$ we have that
\begin{equation}
    \lim_{(i,j)\to \infty} C^i \epsilon_{i,j}=0, \label{eq:epsilon decay req}
\end{equation}
and we will let $(\mathcal{A}_{i,j})_{(i,j)\in\Delta}$ be a sequence of finite collections of blocks in $\mathbb{N}^*$ satisfying that $\mathcal{A}_{i,j}\subseteq \mathcal{A}_{i,j+1}$ and that $\bigcup_{(i,j)\in\Delta} \mathcal{A}_{i,j}=\mathbb{N}^*$. These sequences will be fixed regardless of our choice of $z$.

For each $(i,j)$, we will construct a block $\overline{B}_{i,j}$ based on the prior blocks $\overline{B}_{1,1},\dots,$ $ \overline{B_{i,j-1}}$ and on $\epsilon_{i,j},\epsilon_{i,j+1},\mathcal{A}_{i,j},\mathcal{A}_{i,j+1}, z(2i-1)$ and $z(2i)$.  (Note: the reason for basing $\overline{B}_{i,j}$ on $\epsilon_{i,j+1}$ and $\mathcal{A}_{i,j+1}$ is to make sure that $\overline{B}_{i,j}$ is long enough to ensure that the subsequent block $\overline{B}_{i,j+1}$ can be constructed with the desired properties.) The CF expansion of $\phi(z)$ will be the concatenation of these blocks in order. 

For convenience, we will denote by $\tilde{B}_{i,j}$ the concatenation of all blocks up to (but \emph{not} including) $\overline{B}_{i,j}$. We will let $\tilde{N}_{i,j}$ denote $|\tilde{B}_{i,j}|$.

To begin with, we will simply assume that $\overline{B}_{1,1}=(1)^{N_0}$, for some positive integer $N_0$ that is sufficiently large. We will describe later the exact properties we wish $N_0$ to satisfy, but note here that this can be selected to be independent of $z$.

Now that we have our initial block, we will proceed with the iteration. Given $\tilde{B}_{i,j}$ we want to construct $\overline{B}_{i,j}$. We will drop the cumbersome subscripts and let $\tilde{B}$ just denote $\tilde{B}_{i,j}$. Likewise, we let $b=b_{i,j}$, $\tilde{N}=\tilde{N}_{i,j}$, $\epsilon=\epsilon_{i,j}$, and $\mathcal{A}=\mathcal{A}_{i,j}$. 
Also, let $n=\lfloor \sqrt{\tilde{N}}\rfloor$.

\subsection{Constructing a block $B$}

In order to construct $\overline{B}$, we will first construct a block $B$. The block $B$ will represent an initial ideal choice with good normality properties, both in the base-$2$ expansion and continued fraction expansion. In the next section we will alter $B$ to get $\overline{B}$, in order to potentially break normality as $z$ dictates.

The block $B$ will be based on an integer $N$ that is chosen to be sufficiently large to make the desired inequalities later in the proof hold true. We will also need a variable $m$ associated to $N$, which is given by  $m:= \lfloor \exp((\lambda_{\text{KL}}+\epsilon)N)\rfloor$. We will think of $N$ as a count for a number of digits, and $m$ as a denominator size, but we note that the typical fraction with denominator $m$ would be expected to have slightly more than $N$ digits.

For the most part the inequalities we want $m$ and $N$ to satisfy will be direct, and the ability to choose $N$ to satisfy them will be clear. There are a few exceptions which we point out here. First, we require that $N$ is large enough so that
\begin{align}
    &8L \log 2 (5(\lambda_{\text{KL}}+\epsilon_{i,j+1})N+\log_{\phi} b_{i,j+1}+1)
     \xi(\epsilon_{i,j+1},\mathcal{A}_{i,j+1}) \times\\ & \qquad \times\exp\left(-\eta(\epsilon_{i,j+1},\mathcal{A}_{i,j+1}) \frac{\lfloor \sqrt{N}\rfloor}{\log \lfloor \sqrt{N}\rfloor} \right)\le \frac{(b_{i,j+1}-1) }{10b_{i,j+1}}, \label{eq:xieta next step bound}    \end{align}
where $L$ was the bound derived from Renyi's condition. (We also want to choose $N_0$ large enough so that the above equation holds with $N=N_0$ and $i=j=1$.) Second, we require that $N$ is large enough so that 
\begin{equation}
Ci 2^{-\delta \sqrt{N}} \le \frac{(b_{i,j+1}-1)}{5b_{i,j+1}}, \label{eq:condition (E) at next step}
\end{equation}
where $C$ is the implicit constant in Lemma \ref{lem:Copeland-Erdos fraction variant} and $\delta= \delta(\epsilon_{i,j+1},j+1,2)$; moreover, we want $N$ to be sufficiently large that when, in the context of Lemma \ref{lem:Copeland-Erdos fraction variant}, we take $m$ to be $\lfloor \sqrt{N}\rfloor$, $k$ to be either $i$ or $i+1$, and $\epsilon$ to be $\epsilon_{i,j+1}$, then $m$ satisfies the ``sufficiently large" condition for the lemma to apply. Third, we assume that
\begin{equation}
    \left\lfloor \sqrt{N}\right\rfloor \ge 2\left(\max_{A\in \mathcal{A}_{i,j+1}} |A|+1\right). \label{eq: Lemma 4.2 cond at next step}
\end{equation}
Finally, we will also assume that $N\ge i \tilde{N}_{i,j}$.

We will eventually show that when we construct $\overline{B}_{i,j}$, we get that the length of $\tilde{B}_{i,j+1}$ is at least the value of $N$ chosen above. In particular, inequalities \eqref{eq:xieta next step bound} and \eqref{eq:condition (E) at next step} may be assumed to hold with $N$ replaced with $\tilde{N}_{i,j+1}$ for all values of $(i,j)\in \Delta$.


We want to find a block $B$ that satisfies the following conditions:
\begin{enumerate}
    \item[(A)] Let $d$ be the smallest power of $b$ such that $m^2<d$. (Note that $d/b\le m^2<d$, which will match our condition in Lemma \ref{lem:splitter}.) Then $q(\tilde{B}B)=d$.
    \item[(B)] Let $B'$ be the prefix of $B$ such that $r_{\tilde{B}B'}=r_{\tilde{B}B}(m)$. Then $B'$ is the concatenation of an $(\epsilon,\mathcal{A},n)$-normal, $(\epsilon,\lfloor \sqrt{N}\rfloor)$-denominator normal block of length at least $ N-\tilde{N}$ and a block  of length at most $5$.
    \item[(C)] Let $B''$ be the prefix of $B^*$ such that $r_{B''} = r_{(\tilde{B}B)^*}(m)$. Then $(B'')^*$ is the concatenation of a block of length at most $5$ and an $(\epsilon,\mathcal{A},\lfloor \sqrt{N}\rfloor)$-normal, $(\epsilon,\lfloor \sqrt{N}\rfloor)$-denominator normal block of length at least $N$.
    \item[(D)] If $b=2$, then let $a/d= r_{\tilde{B}B}$, and let $c_0,c_1$ be strings such that $|c_0| = L_b(C_{\tilde{B}})$, and $c_0c_1=a$ , when $a$ is seen as an appropriate binary string. Then we have that $c_1$ is $(\epsilon,i,n)^*$-normal.
    \item[(E)] If $b\neq 2$, then let $M=d^i$ and let $j'=\lceil \log_2 d \rceil-L_2(C_B)$. Let $c_0c_1\dots c_t=S_2(C_{\tilde{B}BM})$ be such that $|c_0|= L_2(C_B)$,  $|c_1|=|c_2|=|c_3|=\dots =|c_{t-1}|=j'$, $|c_t|\le j'$. Then all of $c_1,c_2,\dots, c_{t-1}$ are $(\epsilon,i,n)^*$-normal.
\end{enumerate}
To find this block, we will count the total number of blocks that satisfy all of these conditions, and show that it is greater than $1$. We will then let $B$ be the smallest (lexicographically) such block, where we define which block we are choosing merely to make a consistent choice that will remain constant when large-index values of $z$ change, so that the function $\phi$ is continuous.

Note, moreover that the choice of $m$ will be larger than $q(\tilde{B})$, so, say, $B'$ being chosen to satisfy $r_{\tilde{B}B'}=r_{\tilde{B}B}(m)$ is well-defined.

First, note that the number of fractions of the form $a/d$ in lowest terms in $C_{\tilde{B}}$ is equal to the number of all fractions of the form $a/d$ in $C_{\tilde{B}}$ minus the number of fractions of the form $ab/d$ in $C_{\tilde{B}}$, since we chose $d$ to be a power of $b$. In other words, this is
\[
\left( d\cdot \lambda(C_{\tilde{B}})+O(2)\right) - \left( \frac{d}{b}\cdot \lambda(C_{\tilde{B}})+O(2)\right) =\frac{d(b-1)\lambda(C_{\tilde{B}})}{b}+O(4),
\]
where the implicit constant is $1$. This counts the number of blocks that satisfy condition (A).

For conditions (B) and (C), we will shift the question to the set of Farey fractions. We start with (B). Consider a block $B$ satisfying condition (A) and let $B'$ be the prefix of $B$ such that $r_{\tilde{B}B'}=r_{\tilde{B}B}(m)$.  Let $P/Q$ be the largest fraction in $\mathcal{F}_m$ that is less than $r_{\tilde{B}B}$, and let $\tilde{B}B_p$ be the penultimate prefix of the shortest way of writing the CF expansion of $P/Q$ (``p" for penultimate or prefix). By Lemma \ref{lem:compare point to Farey} and Remark \ref{rem:Farey interval to CF interval}, $B'$ is the concatenation of $B_p$ and a block of length at most $5$. This is not necessarily true if $P/Q$ lies outside $C_{\tilde{B}}$, because then it will not be possible to write $P/Q$ in the desired way; however, there is at most one such $P/Q$, and by Proposition \ref{prop:total points in Fareys}, the number of $a/d$ that have this as a preceding fraction is at most
\[
4 d \left( r_{\tilde{B}B} - \frac{P}{Q}\right) = \frac{4d}{Qq(\tilde{B}B)} \le \frac{8d}{m},
\]
where we have twice the constant from that proposition due to $a/d$ having two possible CF expansions. And this bound is negligible compared to the number of terms satisfying condition (A) if we choose $N$ (and hence $m$) sufficiently large, so we will assume any such $B'$ that arises in this way fails condition (B). Now, if $B_p$ is $(\epsilon,\mathcal{A},n)$-normal, $(\epsilon,\lfloor \sqrt{N}\rfloor)$-denominator normal, and $|\tilde{B}B_p|\ge N$, then it is clear that $B$ satisfies condition (B). Let $U$ denote the set of $P/Q\in \mathcal{F}_m\cap C_{\tilde{B}}$ such that the corresponding $B_p$ is either non-$(\epsilon,\mathcal{A},n)$-normal, non-$(\epsilon,\lfloor \sqrt{N}\rfloor)$-denominator normal, or $|\tilde{B}B_p|<N$. Then, by Proposition \ref{prop:total points in Fareys} again, the total number of blocks $B$ satisfying condition (A) but failing condition (B) is at most
\begin{equation}\label{eq:cond A but not B bound}
4 d\cdot \lambda\left( \bigcup_{P/Q\in U} \left\{ x\middle| \frac{P}{Q}\le x\le \frac{P'}{Q'}\right\}\right).
\end{equation}
The bound is twice what was in the proposition due to a rational number having two possible CF expansions.

By Lemma \ref{lem:BV variant}, the number of fractions $P/Q$ in $\mathcal{F}_m\cap C_{\tilde{B}}$ whose corresponding $\tilde{B}B_p$ has length strictly less than $N$ is $O\left(
\frac{m^2}{\sqrt{N}}\right)
$.
By Lemma \ref{lem:AB Lemma 4}, we see that the contributions of such terms to \eqref{eq:cond A but not B bound} is bounded by
\[
O\left( d\cdot \frac{1}{m} \sqrt{\frac{m^2}{\sqrt{N}}} \right) =O\left(  \frac{ d}{N^{1/4}}\right).
\]

Let $\mathcal{B}$ denote the set of all blocks $B_p$ arising in the above way that are \emph{not} $(\epsilon,\mathcal{A},n)$-normal. Then, applying \eqref{eq:Renyi-condition} and Lemma \ref{lem:Scheerer-extended} (and making use of \eqref{eq: Lemma 4.2 cond at next step} to ensure the conditions of the lemma are met), we have that
\begin{align}
\mu\left( \bigcup_{B_p\in \mathcal{B}} C_{\tilde{B} B_p}\right) &\le L \mu\left( C_{\tilde{B}}\right) \mu\left(\bigcup_{B_p\in \mathcal{B}} C_{ B_p} \right) \\
&\le L \mu\left( C_{\tilde{B}}\right) \xi \exp\left( - \eta \frac{n}{\log n} \right),\label{eq:CtildeBB bound}
\end{align}
where $\xi=\xi(\epsilon,\mathcal{A})$ and $\eta=\eta(\epsilon,\mathcal{A})$ are as in Lemma \ref{lem:Scheerer-extended}. Thus, remembering that Lebesgue and Gauss measure are within a constant multiple of each other, the contribution to \eqref{eq:cond A but not B bound} arising from such blocks is 
\begin{equation}
\le 8L\log 2 \cdot  d \cdot \lambda \left( C_{\tilde{B}}\right) \xi \exp\left( - \eta \frac{n}{\log n}\right). \label{eq:cond A but not B bound 2}
\end{equation}
By \eqref{eq:xieta next step bound}, this is no more than \[
\frac{d(b-1)\lambda(C_{\tilde{B}})}{10b}.
\]

Next, note that if $r_{\tilde{B}B'}\in \mathcal{F}_m$, then we also have that $r_{B'}\in\mathcal{F}_m$, and the number of the latter fractions that are non-$(\epsilon,\lfloor \sqrt{N}\rfloor)$-denominator normal is at most $O(m^2\log N/\sqrt{N})$ by Proposition \ref{prop:bound on denominator normal}.  By Lemma \ref{lem:AB Lemma 4} again, the contribution to \eqref{eq:cond A but not B bound} is $O(d\sqrt{\log N}/N^{1/4})$.

Combined, we see that we can choose $N$ large enough so that the total number of blocks that satisfy condition (A) but not condition (B) is at most 
\[
\frac{d(b-1)\lambda(C_{\tilde{B}})}{5b}.
\]

The method of counting blocks that satisfy condition (A) but not condition (C) proceeds in a similar fashion. The major differences are as follows. First, it is now the last $5$ digits of $B''$ and hence the first $5$ digits of $(B'')^*$ that are the ones we have no control over. Second, when we bound the number of non-$(\epsilon,\lfloor \sqrt{N}\rfloor)$-denominator normal blocks, we need to use not only Proposition \ref{prop:bound on denominator normal}, but the remark following it as well.  Finally, we want to apply Lemma \ref{lem:Scheerer reversed} instead of Lemma \ref{lem:Scheerer-extended}, with $K$ equal to the maximum possible length of a CF expansion with denominator $d$. As is implied in the proof of Lemma \ref{lem:CF to base b digits}, this gives $K=\lceil \log_\phi d\rceil$. By the various relations between $d,m,N$, we see that
\[
K\le \log_\phi (m^2 b) + 1 \le 5(\lambda_{\text{KL}}+\epsilon)N+\log_\phi b + 1.
\]
This results in a bound of size
\[
O\left( d N \xi \exp\left( - \eta \frac{\lfloor \sqrt{N}\rfloor}{\log \lfloor \sqrt{N}\rfloor}\right)\right),
\]
which lacks the $\lambda(C_{\tilde{B}})$ factor from  \eqref{eq:cond A but not B bound 2}, has an additional $N$ factor, and also uses  $\lfloor \sqrt{N}\rfloor$ in place of $n$. However, by \eqref{eq:xieta next step bound}, we see that by choosing $N$ sufficiently large, we can bound all the blocks that satisfy condition (A) but not condition (C) by
\[
\frac{d(b-1)\lambda(C_{\tilde{B}})}{5b}
\]
again.

We pause before continuing on and note that due to Lemma \ref{lem:splitter}, we do not necessarily have that a block $B$ satisfying conditions (A), (B), and (C) can be written as the concatenation of one $(\epsilon,S,n)$-normal block of length $N$ and one $(\epsilon,S,\lfloor\sqrt{N}\rfloor)$-normal block of length $N$. Instead, we can write $B$ as a concatenation of an $(\epsilon,S,n)$-normal block of length  at least $N-1$, a block of bounded length, and another $(\epsilon,S,\lfloor \sqrt{N}\rfloor)$-normal block of length at least $N$. (The reason for the $N-1$ bound on the first block is that if $|B'|+|B''|-|B|=1$, then we will want to consider the penultimate prefix of $B'$ rather than the entirety of it as an $(\epsilon,S,n)$-normal block.)

Now consider the number of blocks which satisfy condition (A) but not condition (D). It is possible that $c_0$ is not consistent over all such blocks but it can take at most one of 3 separate values by Remark \ref{rem:base-b digit variants}. The number of possible $c_1$ that can occur and not be $(\epsilon,i,n)$-normal is at most $O(2^{|c_1|-\delta n})$ by Lemma \ref{lem:Copeland-Erdos var}, with $\delta=\delta(\epsilon,i,b)$.  If we multiply the total number of non-$(\epsilon,i,n)$-normal $c_1$'s by the number of times they can occur (one for each possibility for $c_0$), we get that there are still at most $O(2^{|c_1|-\delta n})$ blocks which satisfy condition (A) but not condition (D).  Note that $|c_0|=L_2(C_{\tilde{B}})= -\log_2 \lambda(C_{\tilde{B}})+O(1)$ and $|c_0c_1|= \log_2 d $. So therefore, in this case, we have that the number of blocks which satisfy condition (A) but not condition (D) is at most
\[
O\left(  \frac{\lambda(C_{\tilde{B}}) d}{2^{\delta n}} \right).
\]
Since the implicit constants here are uniform,  we will make sure we have chosen $N_0$ large enough at the initial stage and each $N$ at the iterative stage, so that this will always be at most
\[
\frac{d(b-1)\lambda(C_{\tilde{B}})}{5b}.
\]

When $b\neq 2$, we must count how many blocks satisfy condition (A) but not condition (E). This is, however, completely answered by Lemma \ref{lem:Copeland-Erdos fraction variant}. Namely, provided $N$ is large enough, there are at most $Ci \lambda(C_{\tilde{B}}) d 2^{-\delta n}$ such blocks, where $\delta=\delta(\epsilon,j,2)$ and $C$ is the implicit constant in the lemma. However by \eqref{eq:condition (E) at next step}, we see that this is always at most 
\[
\frac{d(b-1)\lambda(C_{\tilde{B}})}{5b}
\]
again.

We thus see from the above that the number of blocks that satisfy condition (A) but not condition (B) (or (C) or (D) or (E)) is at most $4/5$th of the number of blocks that satisfy just condition (A). So, provided $N$ is large enough, this tells us there must be at least one block that satisfies all the given conditions. Again, we will chose the lexicographically smallest such block to call $B$ from here on.

\subsection{Constructing $\overline{B}$ from $B$}
We will  construct $\overline{B}$ from $B$ in two very different ways depending on whether $b=2$ or not. 

Let us first consider the case where $b=2$. In this case we will pick $\overline{B}$ to be the longest prefix of $B$ such that 
\begin{equation}\label{eq:bound for overline-B}
L_2(C_{\tilde{B}\overline{B}}) < \left(\log_2 d \right)\left( 1+ \frac{1}{z(2i)+1} \right).
\end{equation}
We claim that, in fact,
\[
L_2(C_{\tilde{B}\overline{B}}) = \left(\log_2 d \right)\left( 1+ \frac{1}{z(2i)+1} +O(\epsilon)\right).
\]

\begin{proof}[Proof of claim]
We note that if we let $B'$ be the prefix of $B$ from condition (B) above, then  
\begin{align*}
L_2(C_{\tilde{B}B'})&= -\left\lceil \log_2 \lambda(C_{\tilde{B}B'}) \right\rceil \le - \left\lceil \log_2 \left( \frac{1}{m^2}\right) \right\rceil\\
&\le \log_2 (m^2)  = \log_2(d).
\end{align*}
Therefore, at its shortest, $\overline{B}$ is simply $B'$.

\begin{rem}
By condition (B), this implies that in the case $b=2$, $\overline{B}_{i,j}$ has length at least $N-\tilde{N}$ and hence $\tilde{B}_{i,j+1}$ has length at least $N$, as desired. 
\end{rem}

By conditions (B) and (C), we have very fine control over the growth of the denominators in $B'$ and $B''$, but as mentioned above, $B$ could be formed as a concatenation $B_1 B_2 B_3$, where $B_1, B_3$ have the desired normality properties (namely $(\epsilon,\lfloor\sqrt{N}\rfloor)$-denominator normality) and $B_2$ has uniformly bounded length.  However, it could be that $B_2$ or the very end of $B_1$ or $B_3$ has extremely large digits and we wish to show now that this will not be the case. We decompose $B_1$ as $B_1'B_1''$ where $B_1'$ is the longest prefix of $B_1$ whose length is a multiple of $\lfloor \sqrt{N}\rfloor$, and likewise decompose $B_3$ as $B_3' B_3''$ where again $B_3'$ is the longest prefix of $B_3$ whose length is a multiple of $\lfloor \sqrt{N}\rfloor$. We then let $B_2'= B_1''B_2$ and $B_4'=B_3''$, so that $B=B_1'B_2'B_3'B_4'$. Apply Proposition \ref{prop:product of denominators} multiple times, we have that
\begin{align*}
d^{-2} &\asymp \lambda(C_{\tilde{B}B}) \asymp \lambda(C_{\tilde{B}}) \lambda(C_{B_1'})\lambda(C_{B_2'})\lambda(C_{B_3'})\lambda(C_{B_4'}) \\
&\asymp \frac{\lambda(C_{\tilde{B}}) }{q(B_1')^2q(B_2')^2q(B_3')^2q(B_4')^2}.
\end{align*}
After rearranging, we get
\[
q(B_2')q(B_4') \ll \frac{\sqrt{\lambda(C_{\tilde{B}})} \cdot d}{q(B_1')q(B_3')}.
\]
By our assumptions on the denominator-normality of $B_1'$ and $B_3'$, we have that 
\begin{align*}
q(B_1')q(B_3')&\ge \exp\left( (N-\tilde{N}-\sqrt{N})(\lambda_{\text{KL}}-\epsilon)\right) \cdot \exp\left( (N-\sqrt{N})(\lambda_{\text{KL}}-\epsilon)\right)\\
&= \exp\left( 2 (N-\sqrt{N})(\lambda_{\text{KL}}-\epsilon)\right) \exp(-\tilde{N}(\lambda_{\text{KL}}-\epsilon)).
\end{align*}
At the same time, we have that
\begin{align*}
d&\le 2m^2\le 2\exp\left( 2(\lambda_{\text{KL}}+\epsilon)N)\right)
\end{align*}
As a result, we have that
\[
q(B'_2)q(B'_4) \ll \sqrt{\lambda(C_{\tilde{B}})} \exp\left( 4\epsilon N +\tilde{N}(\lambda_{\text{KL}} -\epsilon)+2\sqrt{N}(\lambda_{\text{KL}} -\epsilon) \right),
\]
and provided $N$ is sufficiently large, we get that
\[
q(B'_2) q(B'_4)\le \exp(5\epsilon N).
\]

Consider $B_3'$ again. Let $q$ be any fixed positive integer less than $q(B_3')$, and let $\ell$ be the largest integer such that 
\[
\exp\left( \ell \lfloor \sqrt{N}\rfloor (\lambda_{\text{KL}} +\epsilon)\right) \le q.
\]
If we let $B_{3,p}$ be the prefix of $B_3'$ such that $|B_{3,p}|=\ell\lfloor \sqrt{N}\rfloor$, then by the $(\epsilon,\lfloor\sqrt{N}\rfloor)$-denominator normality of $B_3'$, we see that 
\[
\frac{q}{q(B_{3,p})} \le \frac{\exp\left(( \ell+1) \lfloor \sqrt{N}\rfloor (\lambda_{\text{KL}} +\epsilon)\right)}{\exp\left( \ell \lfloor \sqrt{N}\rfloor (\lambda_{\text{KL}} -\epsilon)\right)}\le \exp\left( \lfloor \sqrt{N}\rfloor (\lambda_{\text{KL}}+\epsilon) +2\ell \lfloor \sqrt{N}\rfloor \epsilon\right).
\]
By the properties outlined in condition (C) above, we must have that $q(B_3')\le m$, and thus that $\ell\le N/\lfloor\sqrt{N}\rfloor$. So therefore, by choosing $N$ sufficiently large we can bound $q/q(B_{3,p})$ by $\exp(3\epsilon N)$ for all $q$.

By the above two paragraphs, combined with Proposition \ref{prop:product of denominators}, we see that for any $q$ between $q(B')$ and $q(B)$, we can choose a prefix $B_p$ of $B$ satisfying 
\[
q \ge q(B_p)\ge q\cdot \exp(-8\epsilon N),
\]
provided $N$ is sufficiently large. Recalling \eqref{eq:bound for overline-B}, and our correspondence between $L_2(C_\cdot)$ and $q(\cdot)$, we see that our desired block $\overline{B}$ satisfies 
\[
L_2(C_{\tilde{B}\overline{B}}) = (\log_2 d) \left( 1+ \frac{1}{z(2i)+1}+O(\epsilon)\right),
\]
as desired.
\end{proof}

Now recall that $a/d=r_{\tilde{B}B}$. So therefore, $a/d$ must belong to the cylinder set $C_{\tilde{B}\overline{B}}$. In particular, the binary expansion of any element in $C_{\tilde{B}\overline{B}}$ must consist of the $\log_2 d$ digits of $a$ (or $a-1$), followed by a string of $(\log_2 d)((z(2i)+1)^{-1}+O(\epsilon))$ repeating $0$'s or $1$'s. This will be important in the next section.

Now consider the case where $b\neq 2$. 
In this case, we define $\overline{B}:=BM(1)^K$, where $M=d^i$ and 
\[
K= \left\lceil \frac{|\tilde{B} B|}{z(2i-1)} \right\rceil.
\]
Note that we are including $K$ copies of the single digit 1. (Again, this implies that $\tilde{B}_{i,j+1}$ will have length at least $N$ as desired.)

\subsection{Checking the construction}

Having constructed all of the $\overline{B}$'s following the method of the previous two sections, it remains to show that the resulting infinite sequence gives us the CF expansion of a number $\phi(z)$ with the desired properties. To do this we will use Proposition \ref{prop:normality test} repeatedly.

First consider any base $b$ that is not a power of 2. We claim that the resulting number $\phi(z)$ is not normal to any such base, regardless of $z$. Let $p$ be the smallest odd prime factor of $b$. By the construction of $\phi(z)$ above, we see that for each $i,j$ with $b_{i,j}=p$, there is some $a,f\in\mathbb{N}$ such that 
\[
x= \frac{a}{p^f}+O\left(\frac{1}{p^{if}}\right)
\]
by conditions (A) and (E) in the construction of $\overline{B}_{i,j}$. We thus see that
\[
b^f x \equiv O\left( \frac{(b/p)^f}{p^{(i-1)f}}\right) \pmod{1}.
\]
Let us suppose that $i$ is sufficiently large so that $(b/p)<p^{(i-1)/2}$, so that
\[
b^f x \equiv O\left( \frac{1}{p^{f(i-1)/2}} \right) \pmod{1}
\]
In other words, starting from the $f$th base-$b$ digit of $\phi(z)$, and continuing for the next $f(i-1)\log_b(p)/2+O(1)$ digits, the digits are all either $0$ or $(b-1)$. Either $0$ or $(b-1)$ must be the repeated digit infinitely often. Suppose it is the $0$'s. (The other case is treated similarly.) Then we may consider the base-$b$ expansion of $\phi(z)$ to be broken up as a concatenation $w_1v_1w_2v_2w_3v_3\dots$ as in Proposition \ref{prop:normality test}, where the $v_i$'s are all of these long blocks of $0$'s.  Since $i$ can be taken arbitrarily large, the proposition clearly implies that $\phi(z)$ cannot be base-$b$ normal. 

Now we want to show that $\phi(z)$ is CF-normal if and only if $z\in C$. Suppose first that $z\in C$. Each block $\overline{B}_{i,j}$ can be decomposed as 
\[
\overline{B}_{i,j}= B_{i,j,1} B_{i,j,1}' B_{i,j,2} B_{i,j,2}',
\]
with $B_{i,j,1}$  $(\epsilon_{i,j},\mathcal{A}_{i,j},n_{i,j})$-CF-normal and $B_{i,j,2}$  $(\epsilon_{i,j},\mathcal{A}_{i,j},\lfloor \sqrt{N_{i,j}}\rfloor)$-CF-normal. Moreover, $|B_{i,j,1}'|$ is uniformly bounded over all blocks (see the above remark after bounding how many blocks satisfy condition (A) but not condition (C)), and 
\[B_{i,j,2}'=\begin{cases}\wedge, & b=2,\\ M(1)^K, & b\neq 2.\end{cases}\] 
Applying Proposition \ref{prop:normality test}, we think of the CF expansion of $\phi(z)$ as an alternating concatenation of $w$'s and $v$'s, with the $v$'s consisting of $B_{i,j,1}'$'s and $B_{i,j,2}'$'s (as well as our initial block) and the $w$'s consisting of $B_{i,j,1}$'s and $B_{i,j,2}$'s. Our assumption that $z\in C$ guarantees that $|B_{i,j,2}'|$ is eventually smaller than any (small) positive constant times $|B_{i,j,1}|+|B_{i,j,2}|$ as well. This guarantees that condition \eqref{eq:normality test bound v's} holds. The $\epsilon$'s and $\mathcal{U}$'s satisfy the desired properties clearly. We note that each $m$ is equal to either $n_{i,j}$ (when the corresponding $w$ equals $B_{i,j,1}$) or $\lfloor \sqrt{N_{i,j}}\rfloor$ (when the corresponding $w$ equals  $B_{i,j,2}$). In particular, the $m$'s are strictly increasing and $m_i \le \sqrt{|w_1\dots w_{i-1}|}$. So therefore
\begin{equation} \label{eq:confirming normality test bound}
\frac{\sum_{j\le i} m_i}{|w_1w_2\dots w_{i-1}|} \le \frac{i m_i}{|w_1w_2\dots w_{i-1}|} \le \frac{i}{\sqrt{|w_1w_2\dots w_{i-1}|}}.
\end{equation}
So condition \eqref{eq:normality test bound} holds if $\lim_{(i,j)\to \infty} i^2/\sqrt{\tilde{N}_{i,j}} = 0 $. (Here we are making use of the fact that the first $\tilde{N}_{i,j}$ digits of $\phi(z)$ should account for $O(i^2)$ of the first $w$'s.) However, the condition $N\ge i \tilde{N}$ guarantees that $\tilde{N}$ grows at least exponentially, so this is satsified, and thus $\phi(z)$ is CF-normal.

Conversely, suppose $z\not\in C$. Then we decompose $\overline{B_{i,j}}$ when $j\neq 1$ as 
\[
\overline{B}_{i,j} = B_{i,j,1} B_{i,j,1}'
\]
where $B_{i,j,1}'=(1)^K$. By our assumption that $z\not\in C$, there exists a constant $\epsilon>0$ such that $|B_{i,j,1}'|> \epsilon |\tilde{B}_{i,j}B_{i,j,1}|$ for infinitely many $(i,j)\in\Delta$. Since $B_{i,j,1}$ consists only of $1$'s, we may decompose the CF expansion of $\phi(z)$ as $w_1v_1w_2v_2\dots$ with the $w$'s equal to the blocks $B_{i,j,1}$ and the $v$'s equal to the blocks $B_{i,j,1}'$. The first part of Proposition \ref{prop:normality test} shows that $\phi(z)$ is not CF-normal.

Now assuming that $z\in C$, we will show that $\phi(z)$ is base-$2$ normal if and only if $z\in D$. This will suffice to prove the theorem, since if $z\not\in C$, then we already know that $\phi(z)$ does not belong to the set of CF-normal, absolutely abnormal numbers. 

First we claim that if $z\not\in D$, then $\phi(z)$ is not base-$2$ normal. In particular, if $z\not\in D$, then there is some value, call it $g$, that occurs infinitely often along the even indices of $z$. In particular, there is an infinite number of $i$'s, such that by our construction of $\overline{B}$ in the case $b=2$, we see that after the first $\log_2(d)$ digits, there is a string of length $\log_2(d) \left((g+1)^{-1}+O(\epsilon) \right)$ that consists solely of $0$'s or solely of $1$'s. We now wish to apply Proposition \ref{prop:normality test}, decomposing the base-$2$ expansion of $\phi(z)$ as $w_1v_1w_2v_2\dots$, with the $v$'s being these strings of repeated $0$'s or $1$'s (whichever one occurs infinitely often). Since $\epsilon$ tends to zero, \ref{prop:normality test} immediately gives that $\phi(z)$ is not base-$2$ normal. 

Alternately, if $z\in D$, then $z(2i)\to \infty$ with $i$. And thus, the length of the string of $0$'s or $1$'s referenced in the previous paragraph eventually becomes negligible compared to the number of digits that preceded them. Moreover, condition (D) guarantees that the remainder of the binary digits associated to the block $\overline{B}_{i,j}$ (when $j=1$) are $(\epsilon_{i,j},i,n_{i,j})$-normal; and condition (E) guarantees that (when $j\neq 1$), the binary digits associated to $B_{i,j}M$ can be divided into $(\epsilon_{i,j},i,n_{i,j})$-normal strings, plus an additional string of length at most $\lceil \log_2 d\rceil$ (coming from $c_t$ in condition (E)). The length of this last string becomes negligible compared to the number of digits that precede them (as $i$ increases). So when we decompose the base-$2$ expansion of $\phi(z)$ as $w_1v_1w_2v_2\dots$ for the purposes of Proposition \ref{prop:normality test 2} (not Proposition \ref{prop:normality test}), we will let the binary strings that appear as $c_t$ in condition (E) be considered as part of the $v$'s in addition to the strings of repeated $0$'s or $1$'s from the last paragraph. Also, since we have $z\in C$, the binary digits associated to the string $(1)^K$ appended at the end of the block $\overline{B}_{i,j}$ when $j\neq 1$ must also be negligible in length (see Lemma \ref{lem:length - adding block of 1s}). So these digits are also treated as part of the $v$'s.
Thus, the $v$'s satisfy \eqref{eq:normality test bound v's}. All the remaining strings will make up our $w$'s: these correponds to the strings $c_1$ (possibly truncated) from condition (D) as well as the strings $c_1,c_2, \dots c_{t-1}$ from condition (E). For condition \eqref{eq:normality test bound}, we may use much the same argument as we did above (see \eqref{eq:confirming normality test bound}); however, we need to be careful on two fronts. First, in the prior argument, a given block $\overline{B}_{i,j}$ decomposed into a bounded number of $w$'s for applying the proposition. Now, due to condition (E), a given block $\overline{B}_{i,j}$ could decompose into $t$ different $w$'s where $t\ll i$ (see the proof of Lemma \ref{lem:Copeland-Erdos fraction variant}). So the first $\tilde{N}_{i,j}$ CF digits of $\phi(z)$ could account for $O(i^3)$ of the first $w$'s. Second, $|w|$ now measures the length of binary strings rather than CF blocks, but due to Lemma \ref{lem:CF to base b digits}, we know that the length of a binary string is at least a constant times the length of the corresponding CF block. Combining these, we see that \eqref{eq:normality test bound} holds provided $\lim_{(i,j)\to \infty} i^3/\sqrt{\tilde{N}_{i,j}} = 0 $, which is true for the same reasons as before. Finally, condition \eqref{eq:normality test bound 2} holds by \eqref{eq:epsilon decay req} and the growth rate on the $N_{i,j}$'s. Thus we get that $\phi(z)$ must be base-2 normal by Proposition \ref{prop:normality test 2}.

This completes the proof.

\section{Proof of Theorem \ref{thm:secondary}}

We consider the case with $b=2$ and $b'=3$. All other cases are similar. We follow the main points of the previous proof and simply remark on where they differ. 

To begin with, in this case, we define $\Delta = \{(i,j):i\in\mathbb{N}, j\in \{1,2\}\}$. The next change we make is to conditions (D) and (E) in selecting the block $B$. We replace these with the following:

\begin{enumerate}
    \item[(D$'$)] If $b=2$, then let $a/d=r_{\tilde{B}B}$. Let $c_0,c_1$ be strings such that $|c_0|=L_2(C_{\tilde{B}})$ and $c_0c_1=a$, when $a$ is seen as an appropriate binary string. Then we have that $c_1$ is $(\epsilon,i,n)^*$-normal. Moreover, let $c'_0c'_1c'_2c'_3=S_3(C_{\tilde{B}B})$ with $c'_0=S_3(C_{\tilde{B}})$, and $c'_1,c'_2$ all have length $J= \lceil \log_3 d \rceil - L_3(C_{\tilde{B}})$ and $|c'_t|\le L_3(C_{\tilde{B}})+O(1)$. Then both of the strings $c'_1, c'_2$ are $(\epsilon,i,n)^*$-normal.
    \item[(E$'$)] If $b=3$, then let $a/d=r_{\tilde{B}B}$. Let $c_0,c_1$ be strings such that $|c_0|=L_3(C_{\tilde{B}})$ and $c_0c_1=a$, when $a$ is seen as an appropriate ternary string. Then we have that $c_1$ is $(\epsilon,i,n)^*$-normal. Moreover, let $c'_0c'_1c'_2c'_3=S_2(C_{\tilde{B}B})$ with $c'_0=S_2(C_{\tilde{B}})$, and $c'_1,c'_2$ all have length $J= \lceil \log_2 d \rceil - L_2(C_{\tilde{B}})$ and $|c'_t|\le L_2(C_{\tilde{B}})+O(1)$. Then both of the strings $c'_1, c'_2$ are $(\epsilon,i,n)^*$-normal.
\end{enumerate}

In the part of the proof where we count how many blocks satisfy condition (A) but not condition (D$'$), we apply the same argument, except we must now also employ Lemma \ref{lem:Copeland-Erdos fraction variant 2} to bound the number of blocks which could give rise to one of $c'_1,c'_2$ being not-$(\epsilon,i,n)^*$-normal. But it is easy to see that this bound can also be taken to be insignificant compared to the size of the number of blocks satisfying condition (A). (In particular, a similar bound was used in the previous proof in counting the number of blocks which satisfy condition (A) but not condition (E).)  We do the same for the number of blocks which satisfy condition (A) but not condition (E$'$), and thus we see that there must exist a block $B$ with the desired properties.

Regardless of the base we use, we then pick $\overline{B}$ to be the longest prefix of $B$ such that
\[
L_b(C_{\tilde{B}B}) < \left( \log_b d\right) \left( 1+ \frac{1}{z(2i-\sigma)+1}\right),
\]
where $\sigma=1$ if $b=2$ and $\sigma=0$ otherwise. Clearly the claimed result that we will get 
\[
L_b(C_{\tilde{B}B}) = \left( \log_b d\right) \left( 1+ \frac{1}{z(2i-\sigma)+1}+O\left(\epsilon\right)\right)
\]
still holds. Note that this means we will never append the digit $M$ or the block $(1)^K$ to the end of $B$ to form $\overline{B}$, as we did in the proof of Theorem \ref{thm:main}.

As a result of this, it is clear for the same reasons as above that if $z\not\in C$, then $\phi(z)$ will not be $2$-normal, and if $z\not\in D$, then $\phi(z)$ will not be $3$-normal. It remains to show that with our alternate construction that we do obtain the appropriate normality if $z\in C$ or $z\in D$.

Since the cases are similar, we will show that if $z\in C$ then $\phi(z)$ is $2$-normal. The new element to consider is the binary digits associated to the  blocks $\overline{B}$ when $b=3$. Since $\overline{B}$ is taken as a prefix of $B$ in this case, we have that $C_{\tilde{B}{B}}\subset C_{\tilde{B}\overline{B}}$. Therefore, by Remark \ref{rem:base-b digit variants}, 
 $S_2(C_{\tilde{B}\overline{B}})$ (up to addition by $1$ or $2$) is a prefix of $C_{\tilde{B}B}$. The corresponding binary digits associated to $C_{\tilde{B}\overline{B}}$ not associated to $C_{\tilde{B}}$ consist of the string $c_1'$, and either a prefix of $c'_2$ or the entirety of $c'_2$ and a prefix of $c'_3$ (possibly these strings might be off by addition of $1$ or $2$ again). By condition (E$'$), $c'_1,c'_2$ are $(\epsilon,i,n)^*$-normal, so following Remark \ref{rem:base-b digit variants} and Lemma \ref{lem: adding does not alter normality}, we can break the binary digits of associated to $C_{\tilde{B}\overline{B}}$ not associated to $C_{\tilde{B}}$ into two $(\epsilon,i,n)$-normal strings that are prefixes of $c'_1$ and $c'_2$, which we can associate with the $w$'s in Proposition \ref{prop:normality test}, and some additional strings (coming from at most $n+2$ digits at the end of $c'_1,c'_2$ and the entirety of $c'_3$) of total length not exceeding $2(n+2)+L_2(C_{\tilde{B}})+O(1)$, which are associated with the $v$'s. Provided $d$ is sufficiently large compared with $L_2(C_{\tilde{B}})$ and $n$, it is clear that the necessary conditions are satisfied. (We note that in other cases, we need a more general form of Lemma \ref{lem: adding does not alter normality}, but this is easy to do.)

\section{Proof of Theorem \ref{thm:extra}}

To begin with, we fix the block $A$, and let $m$ be the floor of $\sqrt{d}\exp(-\log^2 \log d)$. Moreover, let $N$ be a positive integer satisfying $m=\lfloor\exp((\lambda_{\text{KL}}+\epsilon)N)\rfloor $. (We pause to note the difference between what we are doing and the proof of Theorem \ref{thm:main} above. In that proof we began with $N$, derived $m$ and then $d$ from it. In this case, we begin with $d$, and derive $m$ and then $N$ from it. While the relationship between $m$ and $N$ is the same in each case, the relationship between $m$ and $d$ is slightly different here.) Let  $n=\lfloor \sqrt{N}\rfloor$ and let $c=\log_2 e$.

With these variables, suppose a block $B$ contains two $(\epsilon/2,A)$-normal blocks and at most $2n+15+4\lceil\frac{c}{2}\log^2 \log d\rceil$ other digits, and also suppose that $|B|\ge 2N-1$. Then, by Lemma \ref{lem:epsilon-u subwords}, we see that $B$ is 
\[
\left(\frac{\epsilon}{2}+ \frac{(|A|+2)\left( 2n+25+4\lceil\frac{c}{2}\log^2 \log d\rceil \right)}{|B|-|A|-1} ,A \right)-\text{normal}.
\]
Since $N\asymp \log d$ and $n\asymp \sqrt{\log d}$, the second fraction here will tend to $0$ as $d$ increases (assuming $A$ stays fixed). In particular, we may assume $d$ is large enough that the second fraction is at most $\epsilon/2$, so that the above assumptions guarantee $B$ is $(\epsilon,A)$-normal. 

Recall that a block $B$ is said to be $(\epsilon,A,n)$-normal if every prefix of $B$ whose length is a multiple of $n$ is $(\epsilon,A)$-normal. We will now say that a block $B$ is $(\epsilon,A,n;N)$-normal if every prefix of $B$ whose length equals $N+kn$ for some non-negative integer $n$ is $(\epsilon,A)$-normal.

Consider blocks $B$ of CF digits together with following conditions:
\begin{enumerate}
    \item[(A$'$)] $q(B)=d$.
    \item[(B$'$)] Let $B'$ be the prefix of $B$ such that $r_{B'}=r_B(m)$. Then $B'$ is the concatenation of an $(\epsilon/2,A,n;N-1)$-normal block of length at least $N$ and a block of length at most $5$.
    \item[(C$'$)] Let $B''$ be the prefix of $B^*$ such that $r_{B''}=r_{B^*}(m)$. Then $(B'')^*$ is the concatenation of a block of length at most $5$ and an $(\epsilon/2,A,n;N)$-normal block of length at least $N$.
\end{enumerate}

Suppose a block satisfies all of conditions (A$'$), (B$'$), and (C$'$). Then we have that $|B|\ge 2N-1$ by the upper bound of Lemma \ref{lem:splitter}. By the lower bound of Lemma \ref{lem:splitter}, we know that $B$ can consist of at most the digits of $B'$, the digits of $B''$ and at most $5+ 4\lceil\frac{c}{2}\log^2 \log d\rceil$ other digits. However, by the definition of $(\epsilon/2,A,n)$-normality, we know that $B'$ and $(B'')^*$ both consist of a string of length at most $5$, a string of length at most $n$, and an $(\epsilon/2,A)$-normal block. It is possible that $B'$ and $(B'')^*$ are both $(\epsilon/2,A)$-normal blocks that overlap in $B$ at a single digit (again by the upper bound of Lemma \ref{lem:splitter}); however, if this happens, we could just remove the last $n$ digits from $B'$ and consider the remaining prefix of $B'$ as our desired $(\epsilon/2,A)$-normal block. There will always be digits we can remove and retain an $(\epsilon/2,A)$-normal block, because $|B'|\ge N$, but the smallest $(\epsilon/2,A)$-normal prefix has length $N-1$.

Thus, we get that $B$ can be written as two $(\epsilon/2,A$)-normal blocks, plus at most $2n+15+4\lceil\frac{c}{2}\log^2 \log d\rceil$ other digits.
In other words, these three conditions imply by our above work that $B$ is $(\epsilon,A)$-normal. Therefore, it suffices to bound the number of blocks that satisfy condition (A$'$), but fail either condition (B$'$) or condition (C$'$).

For this we can almost use the exact same estimations as we did in the proof of Theorem \ref{thm:main} to bound the number of blocks that satisfied condition (A) but not conditions (B) or (C). However, there are some differences.

First, by Lemma \ref{lem:Hensley length bound}, the number of possible blocks $B'$ (with no dependence on arising from $B$) of length less than $N$ with $q(B')\le m$ is at most $O(m^2 e^{-c_1 (\epsilon \sqrt{\log m})^2})$ $ = O(d^{1-\epsilon'})$ for some $\epsilon'>0$. Applying Proposition \ref{prop:total points in Fareys} and Lemma \ref{lem:AB Lemma 4} as we did in the proof of Theorem \ref{thm:main}, we see that the number of $B$'s giving rise to a $B'$ or $B''$ of length strictly less than $N$ is at most $O(d^{1-\epsilon'/2})$.

Second, we no longer wish to use Lemmas \ref{lem:Scheerer-extended} or \ref{lem:Scheerer reversed}, but rather a variant of this lemma that measures the size of the set
\[
\mu\left( \bigcup_{k=0}^K E^c_{CF} (\epsilon, A; N-1+kn) \right) 
\]
or the appropriate starred variant.
However, it is easy to adjust the proof of the lemmas to see that these will be bounded by $\xi \exp(-\eta N/\log N)$, where $\xi,\eta$ are dependent only on $\epsilon$ and $A$.

Thus, by the same method as before, the number of blocks $B$ where $B'$ fails to be $(\epsilon/2,A,n;N)$-normal will be at most
\[
O\left(   dN \cdot \xi \exp\left(-\eta\frac{N}{\log N}\right)\right),
\]
and a similar bound will hold on the number of blocks $B$ where $B''$ fails to be $(\epsilon/2,A,n;N)^*$-normal. Since $N\asymp \log d$, this latter bound is far larger than the $O(d^{1-\epsilon'/2})$ seen before. And thus the total number of blocks $B$ that satisfy condition (A$'$) but not conditions (B$'$) or (C$'$) is at most $O(d^{1-\eta'/\log\log d})$ for some constant $\eta'$, as desired.

\bibliographystyle{plain}

\end{document}